\documentclass[12pt,a4paper,reqno]{amsart}
\usepackage{amsmath,amsthm,amsfonts,amssymb,bbm}
\usepackage{graphicx,psfrag,subfigure,color}
\usepackage{cite,dsfont}
\usepackage{hyperref,enumerate}
\numberwithin{equation}{section}

\newcommand{\psiw}{\psi_{\mathtt w}}

\renewcommand{\Re}{\mathrm{Re}}
\renewcommand{\Im}{\mathrm{Im}}

\newtheorem{prop}{Proposition}[section]
\newtheorem{Theorem}[prop]{Theorem}
\newtheorem{Lemma}[prop]{Lemma}
\newtheorem{Definition}[prop]{Definition}

\newtheorem*{MT}{ Main Theorem}

\newtheorem{Remark}[prop]{Remark}

\title[Domino shuffling and stochastic growth]{The domino shuffling algorithm and Anisotropic KPZ stochastic growth}
\author{Sunil Chhita}
\address{Department of Mathematical Sciences, Durham University, Stockton Road, Durham, DH1 3LE, UK.} \email{sunil.chhita@durham.ac.uk} 
\author{Fabio Lucio Toninelli}
\address{Technische Universität Wien, Institut f\"{u}r Stochastik und Wirtschaftsmathematik, Wiedner Hauptstraße 8-10/105-7, A-1040 Wien, Austria}\email{fabio.toninelli@tuwien.ac.at}

\begin{document}
\maketitle

\begin{abstract}
  The domino-shuffling algorithm \cite{EKLP92,Pro03} can be seen as a
  stochastic process describing the irreversible growth of a
  $(2+1)$-dimensional discrete interface
  \cite{CT:18,zhang2018domino}. Its stationary speed of growth
  $v_{\mathtt w}(\rho)$ depends on the average interface slope $\rho$,
  as well as on the edge weights $\mathtt w$, that are assumed to be
  periodic in space. We show that this growth model belongs to the
  Anisotropic KPZ class \cite{Wol91,toninelli20173+}: one has
  $\det [D^2 v_{\mathtt w}(\rho)]<0$ and the height fluctuations grow
  at most logarithmically in time. Moreover, we prove that
  $D v_{\mathtt w}(\rho)$ is discontinuous at each of the (finitely
  many) smooth (or ``gaseous'') slopes $\rho$; at these slopes,
  fluctuations do not diverge as time grows. For a special case of
  spatially $2-$periodic weights, analogous results have been recently
  proven \cite{CT:18} via an explicit computation of
  $v_{\mathtt w}(\rho)$. In the general case, such a computation is out
  of reach; instead, our proof goes through a relation between the
  speed of growth and the limit shape of domino tilings of the Aztec
  diamond.
\end{abstract}
\section{Introduction}
In the realm of stochastic interface growth \cite{BS95}, dimension
$(2+1)$ (i.e., growth of a two-dimensional interface in
three-dimensional physical space) plays a distinguished role. In
$(1+1)$ dimensions, one finds a non-trivial KPZ growth exponent
$\beta=1/3$ as soon as the growth process is genuinely non-linear,
while in dimension $(d+1),d\ge3$ a phase transition is expected
\cite{KPZ86} between a regime of small non-linearity, where the
process behaves qualitatively like the stochastic heat equation (SHE)
with additive noise, and a regime of large non-linearity,
characterized by new growth and roughness critical exponents. See the
recent
\cite{magnen2018scaling,dunlap2018fluctuations,comets2019renormalizing}
for mathematical progress on the small non-linearity regime of the KPZ
equation for $d\ge3$. On the other hand, dimension $(2+1)$ is the
``critical'' or ``marginal'' case: here, the critical exponents are
expected to depend not so much on the intensity of the non-linearity,
but rather on its structure.  In fact, in this case, the existence of
two different universality classes has been conjectured
\cite{Wol91,BS95} (see \cite{toninelli20173+} for a recent
mathematical review). The first, called Anisotropic KPZ (or AKPZ)
class, is characterized by logarithmic growth of height fluctuations
in space and time, like the two-dimensional SHE with additive
noise. The second, called KPZ class \emph{tout court}, has universal
and non-trivial roughness and growth exponents,
$\alpha_{KPZ}\simeq 0.39$ and $\beta_{KPZ}\simeq 0.24$ respectively
(these values are known only numerically,
cf. e.g. \cite{TFW92,HH12}). Conjecturally, the universality class of
a model is determined by the properties of the average speed of growth
$v(\rho)=\lim_{t\to\infty}\frac1t\mathbb E[h(t,x)-h(0,x)]$ of the interface height function $h$, where $\rho$ is the macroscopic slope of the initial condition.  Namely, a model is
expected to belong to the AKPZ class if and only if
$\det(D^2 v(\rho))\le 0$, where  $D^2 v(\rho)$ is 
the $2\times 2$ Hessian matrix.  From the mathematical point of view, the
understanding of the AKPZ universality class has remarkably progressed
lately but it is still limited to a few special cases (see Section
\ref{sec:relw} for references). For the KPZ class, very interesting
recent developments (in a somewhat different direction) concern the
weak non-linearity (or weak-disorder) regime
\cite{chatterjee2018constructing,caravenna2018two}: if non-linearity
is scaled to zero as $\hat \beta/\sqrt{|\log\epsilon|}$, with
$\epsilon\to0$ a noise regularization parameter and provided
$\hat \beta$ is smaller than a precisely identified critical value
$\hat \beta_c$ \cite{caravenna2018two}, then the KPZ equation scales
to the SHE with additive noise. In this regime, the non-trivial exponents
$\alpha_{KPZ},\beta_{KPZ}$ do not emerge.

In the present work, we focus on the so-called ``domino shuffling
algorithm''. This is a discrete-time Markov chain on perfect matchings
(or ``domino tilings'') of $\mathbb Z^2$, that was originally devised
\cite{EKLP92,Pro03} as a way to exactly sample and to count perfect
matchings of certain special two-dimensional domains (Aztec
diamonds). When this algorithm is run on the infinite square grid, it
can be seen also as a $(2+1)$-dimensional growth model, and it is from
this point of view that we consider it here. The shuffling algorithm
is actually an infinite-dimensional family of growth processes,
indexed by the edge weights $\mathtt w$, that we only assume to be
positive and periodic in both lattice directions, with some period
$2n\in 2\mathbb N$.  Along the dynamics, the edge weights also evolve
(deterministically) in time. In fact, the evolution
$\{\mathtt w_k\}_{k\ge0}$ of edge weights under the shuffling
algorithm (or ``spider moves'') has a a remarkable interest in itself, as a
classical integrable dynamical system \cite{GK13}. Its trajectories
are in general not time-periodic.

For generic edge weights of period $2n$, there are $2n(n-1)+1$ special
values for the slope (``smooth'' or ``gaseous'' slopes), that
correspond to ``cusps'' of the surface free energy $\sigma(\rho)$ of
domino tilings with weights $\mathtt w$. The slopes at which $\sigma$
is smooth are instead referred to as ``rough slopes'' (the reason for
the nomenclature smooth/rough is reminded in Section \ref{sec:pwat}).
We let $\mathcal S$ (resp. $\mathcal R$) denote the set of smooth
(resp. rough) slopes.

Our main result is that the domino shuffling algorithm (with general
weights $\mathtt w$) belongs to the AKPZ class, and that the speed of
growth is singular at each of the smooth slopes (see Theorem
\ref{th:1} and Section \ref{sec:fluct} for more precise statements):
\begin{MT}[Informal version] 
  For $\rho\in\mathcal R$, the speed of growth function
  $\rho\mapsto v_{\mathtt w}(\rho)$ is $C^\infty$ and
  $\det[D^2v_{\mathtt w}(\rho)]<0$. On the other hand, the gradient
  $D v_{\mathtt w}(\rho)$ is discontinuous at each of the finitely
  many slopes $\rho\in\mathcal S$. For
  $\rho\in \mathcal R$, the height fluctuations grow logarithmically in
  space (they scale to a Gaussian Free Field) and at most
  logarithmically in time. For $\rho\in \mathcal S$, the variance of
  the height fluctuations is uniformly bounded in space and time.
\end{MT}

In a special case of $2$-periodic weights ($n=1$)
analogous results have been proven recently in \cite{CT:18}. In that
case, there is a single smooth slope ($|\mathcal S|=1$) and the explicit computation of
$v_{\mathtt w}(\rho)$ is doable, though rather involved, via Kasteleyn
theory. In the general case we are considering here, computing
$v_{\mathtt w}(\rho)$ directly using Kasteleyn theory seems very
complicated, and we do not proceed that way.  The first key point in
the proof of the theorem is a simple relation (cf. \eqref{eq:chiev})
between $v_{\mathtt w}(\cdot)$ and the limit shape $\psi_{\mathtt w}$
of the dimer model with edge weights $\mathtt w$ in the Aztec diamond.
The limit shape is nothing but the solution of the Euler-Lagrange
equation \cite{KO07} associated to the dimer model's surface tension,
with weights $\mathtt w$ and boundary conditions determined by the
geometry of the domain.  This relation allows to translate analytic
properties of $v_{\mathtt w}(\cdot)$ into analytic properties of the
limit shapes, for which we use results from \cite{Duse,SS10}.  In
particular, singularities of $v_{\mathtt w}(\cdot)$ are in bijection
with the facets (flat portions) of $\psi_{\mathtt w}$ that do not
touch the boundary of the Aztec diamond or, equivalently, with the
holes of the amoeba of the spectral curve
\cite{Kenyon2003}. In \cite{CT:18}, the discontinuity of
$Dv_{\mathtt w}(\rho)$ at the unique smooth phase was found via the
explicit formula, but the connection with the facet of the limit shape
was not realized.  Another point we wish to emphasize is that, since
edge weights change non-periodically with time as
$\mathtt w=\{\mathtt w_k\}_{k\ge0}$, it is a priori not obvious that an
asymptotic speed of growth even exists (the connection with the limit
shape shows that it does, because the limit shape $\psi_{\mathtt w_k}$
is actually independent of $k$).

Let us conclude this section by mentioning a recent article
\cite{zhang2018domino}, that proves a hydrodynamic limit for the
domino shuffling dynamics, in the form of the convergence of the rescaled 
height profile to the viscosity solution of the non-linear
Hamilton-Jacobi PDE $\partial_t \phi=v_{\mathtt w}(\nabla\phi)$.  The result of
\cite{zhang2018domino} is stated for the case of edge weights with
space periodicity $1$, but the same proof presumably works for general
periodic edge weights, as in the framework of the present article.

\subsection{Related works on AKPZ growth models}
\label{sec:relw}
Historically, the first rigorous result we are aware of, on a
$(2+1)$-dimensional growth model in the AKPZ class, is
\cite{prahofer1997exactly}, that computed the speed of growth of the
Gates-Westcott model \cite{gates1995stationary}, verified that
$\det(D^2 v(\rho))< 0$ and proved that stationary states are only
logarithmically rough, in agreement with the above conjecture (growth
of flucutations in time was not studied there).  More recently, a
growth model that is a $(2+1)$-dimensional, discrete, analog of
Hammersley's process has been introduced in \cite{BF08}. Besides the
computation of the speed of growth and the verification of
$\det(D^2 v(\rho))< 0$, rigorous results on this model include the
proof that height fluctuations grow at most logarithmically in space
\emph{and} time \cite{BF08,toninelli20172+}, the study of stationary
states \cite{toninelli20172+}, hydrodynamic limits for the height
profile \cite{BF08,legras2019hydrodynamic}, determinantal formulas for
certain space-time correlations \cite{BF08} and a CLT on scale
$\sqrt{\log t}$ for height flucutations under special initial
conditions \cite{BF08}.  Some of these results have been extended to
an AKPZ growth process defined in terms of the dimer model on the
square grid, see \cite{chhita2017speed}.

Apart from the above references, that deal with specific models, let
us mention \cite{borodin2018two}, that gives a sufficient condition
for a $(2+1)$-dimensional growth model to belong to the AKPZ class. In simple
terms, \cite[Th. 2.1]{borodin2018two} states that if the hydrodynamic
equation $\partial_t \phi=v(\nabla\phi)$ preserves solutions of the
Euler-Lagrange equations associated to \emph{some} strictly convex
surface tension function $\sigma(\cdot)$, then
$\det(D^2 v(\rho))\le 0$. This condition can be verified on several
growth models, e.g.  the one defined in \cite{BF08}, and it is related
to the fact that these stochastic processes preserve a certain ``local
Gibbs property'' ($\sigma$ is then the surface tension corresponding with such Gibbs potential).

\medskip The rest of the paper is organized as follows.  In
Section~\ref{sec:Model}, we introduce the dimer model on
$\mathbb{Z}^2$ for general weights, give some dimer model theory and
give a precise version of our theorem.  In
Section~\ref{sec:Identification}, we prove the existence of the speed
and its formula, while the main properties of the speed are proven in
Section~\ref{sec:Properties}.

\section{Model and results} \label{sec:Model}

In this section, we introduce the shuffling algorithm for the dimer
model on $\mathbb{Z}^2$ for general weights, some of the basic dimer
model theory and we precisely formulate our results.

\subsection{Shuffling algorithm for the dimer model on $\mathbb{Z}^2$ (general weights)} 

 The vertices of the graph  $\mathbb{Z}^2$ are colored
black and white in a bipartite way and they are assigned Cartesian
coordinates, that is the neighbouring vertices which share a common
edge with the vertex $(0,0)$ are $(1,0), (0,1)$, $(-1,0)$ and
$(0,-1)$.  We label a face  $(i,j)\in \mathbb{Z}^2$ if its
center has coordinates $(i+1/2,j+1/2)$; see Fig.~\ref{fig:coordsfaces}
for an example on a $4\times 4$ torus graph.
\begin{figure}
  \centering
		\includegraphics[height=4cm]{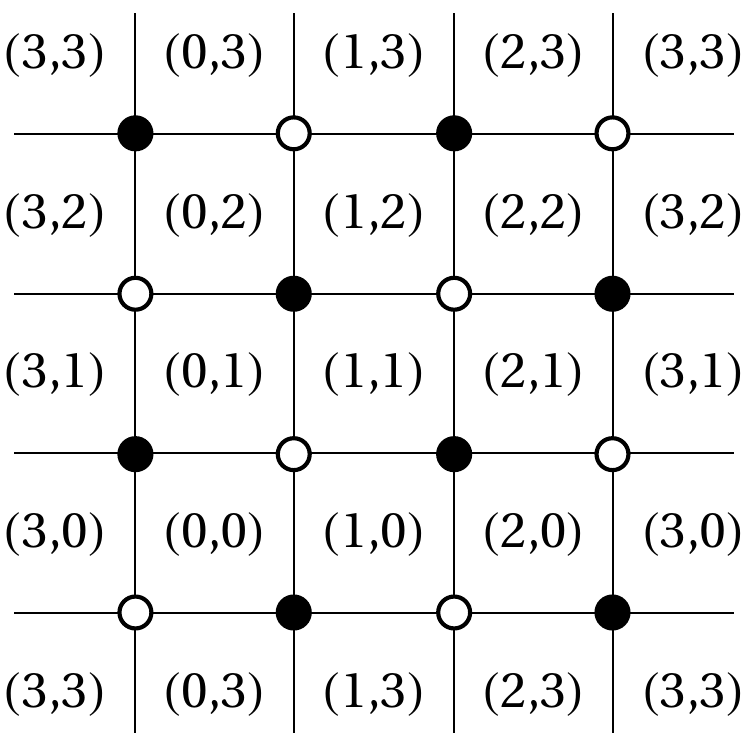}
		\caption{Coordinates of the faces }
              	\label{fig:coordsfaces}
\end{figure}

The discrete time index of the Markov chain will be denoted
$k=0,1,\dots$.  We will say that a face $(i,j)$ is \emph{even} if its bottom-left vertex is white, and \emph{odd} otherwise. In the dynamics
defined below, the colors of the vertices will interchange at each
time step $k$ and we assume that initially the vertex $(0,0)$ is white. Therefore, a face with coordinates $(i,j)$ will be even at time
$k$ if $i+j=k\mod 2$ and odd otherwise.

Given a weighting $\mathtt{w}$ of the edges, i.e. an assignment of a
strictly positive weight to each edge, we first define a deterministic
sequence $\{\mathtt{w}_k\}_{k\ge0}$ of edge weightings with $\mathtt w_0:=\mathtt w$.  To this
purpose, note first that the weighting is uniquely defined if we
specify the weights of edges on the boundary of every even
face. We write then
$$\mathtt{w}_k=\{(w_{i,j;k}^a, w_{i,j;k}^b, w_{i,j;k}^c,w_{i,j;k}^d):
(i,j) \in\mathbb{Z}^2, (i+j)=k\hspace{-2mm}\mod
2\}$$ where the 4-tuple of positive numbers
$(w_{i,j;k}^a, w_{i,j;k}^b, w_{i,j;k}^c, w_{i,j;k}^d)$ denotes the
edge weights around the face $(i,j)$ at time $k$, where $a,b,c$
and $d$ are the edges labelled clockwise around the face, with $a$ being
the topmost horizontal edge on that face.  Also for
$(i,j) \in \mathbb{Z}^2$ and $k \geq 0$, 
set
$$\Delta_{i,j;k}= w_{i,j;k}^a w_{i,j;k}^c +w_{i,j;k}^b w_{i,j;k}^d.$$
The relation between $\mathtt{w}_k$ and $\mathtt{w}_{k+1}$ is, by definition,
\begin{multline}
	\label{eq:wktowk+1}
	(w_{i,j;k+1}^a, w_{i ,j;k+1}^b, w_{i,j;k+1}^c,w_{i,j;k+1}^d)\\:= \left( \frac{w_{i,j+1;k}^a}{\Delta_{i,j+1;k}} ,  \frac{w_{i+1,j;k}^b}{\Delta_{i+1,j;k}}, \frac{w_{i,j-1;k}^c}{\Delta_{i,j-1;k}}, \frac{w_{i-1,j;k}^d}{\Delta_{i-1,j;k}} \right)
\end{multline} 
for $k\geq 0$ and $(i+j)=k+1\hspace{-2mm}\mod 2$; see Fig. \ref{fig:pesi}.
\begin{figure}
	\begin{center}
		\includegraphics[height=6.5cm]{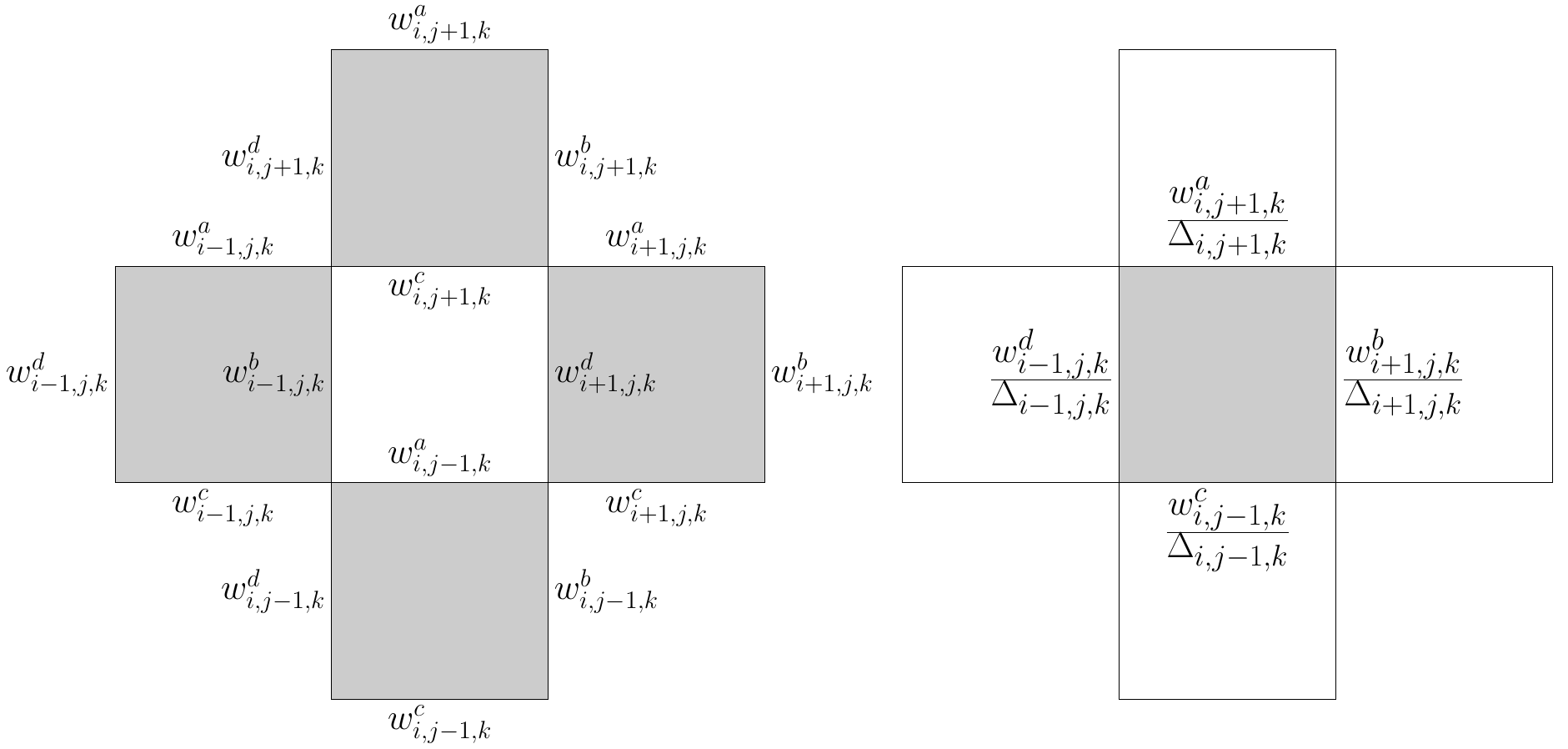}
		\caption{The left figure shows the weights at time $k$ while the right figure shows the weights at time $k+1$ after applying the shuffling algorithm. In each figure, the shaded squares denote the even faces, at times $k$ and $k+1$ respectively. The central face has coordinates $(i,j)$ with $(i+j)=k+1\mod 2$.  }
		\label{fig:pesi}
	\end{center}
\end{figure}

We are now ready to define the shuffling algorithm. This is a
discrete-time Markov chain on $\Omega$, the set of dimer coverings, or
perfect matchings, of $\mathbb{Z}^2$. That is, each $\eta\in\Omega$ is a subset of edges of $\mathbb Z^2$, such that
each vertex is contained in exactly one of them.  Each edge contained
in $\eta$ will be said to be  ``occupied by a dimer''. The chain is not
time-homogeneous, since the transition rates depend on the time index
$k$, via the edge weights $\mathtt{w}_k$.  For $k\geq 0$, we define a
random map $\Omega\ni \eta \mapsto {T_{k+1}}(\eta)\in\Omega$ through the
following four steps, cf. Fig.~\ref{fig:spiderconfigs} (only the third
one is actually random):
\begin{figure}
	\begin{center}
		\includegraphics[height=2cm]{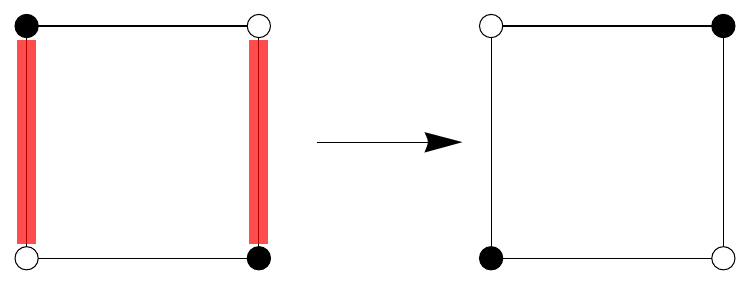}
		\hspace{15mm}
		\includegraphics[height=2cm]{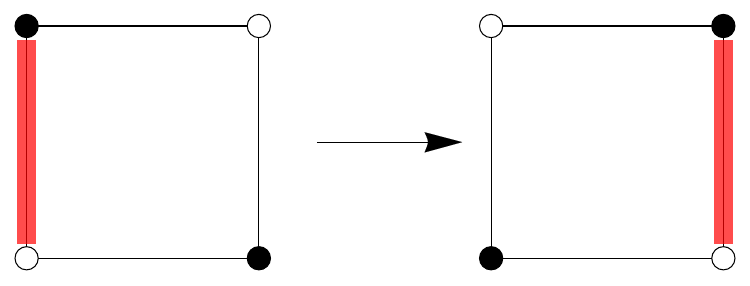}
		\includegraphics[height=2cm]{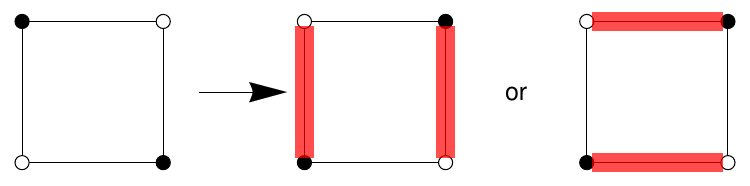}
		\caption{The four steps of the dynamics applied to an even face for the three different possibilites (up to rotations) of boundary edges at that face.}
		\label{fig:spiderconfigs}
	\end{center}
\end{figure}
\begin{enumerate}
\item [(Deletion step)] All pairs of parallel dimers of $\eta$
  covering two of the four boundary edges of any face that is even (at
  time $k$) are removed.
\item [(Sliding step)] For every even face (at time $k$) with only
  one boundary edge covered by a dimer of $\eta$, slide this dimer
  across that face.
\item [(Creation step)] For each face that is even at time $k$ (call $(i,j)$ its coordinates), if there are no dimers of
  $\eta$ covering any of its four boundary edges, add two parallel
  vertical dimers to the face  with probability
		\[ \frac{w^{b}_{i,j;k} w^{d}_{i,j;k}}{\Delta_{i,j;k}}\] or two parallel
  horizontal dimers with probability
  \[ \frac{w^{a}_{i,j;k} w^{c}_{i,j;k}}{\Delta_{i,j;k}}\] (the operations are
  performed independently for each $(i,j)$ and $k$).
\item[(Interchange step)] Interchange the white and black colors of  vertices of
  the graph. 
\end{enumerate}
It is well known, and easy to check, that $T_k(\eta)\in\Omega$ if
$\eta\in\Omega$. 
The swapping of colors at each step, that may seem to be pointless at
this stage, will appear more natural in the discussion below of the
evolution of the height function.

The maps $T_k$ are independent but not identically distributed, since the edge weights depend on $k$. Iteratively applying these maps and
letting \[\eta_k:=T_{k}\circ\dots\circ T_1(\eta_0),\] one obtains the
desired Markov chain $\{\eta_k\}_{k\ge0}$ on $\Omega$.

\subsubsection{Height function and its evolution}
Each dimer configuration $\eta\in \Omega$ is in one-to-one
correspondence (up to a height
offset) with a height function $h_\eta(\cdot)$  which is defined on the faces of $\mathbb{Z}^2$ \cite{KenLectures}. That is, one fixes the height to be zero at some reference face $f_0$ and one defines the height gradients as
\begin{eqnarray}
  \label{eq:hgrad}
  h_\eta(f')-h_\eta(f)=\sum_{e\sim C_{f\to f'}}\sigma_e (\mathds 1_{e\in \eta}-1/4)
\end{eqnarray}
where the sum runs over the edges crossed by a nearest-neighbor path
$C_{f\to f'}$ from $f$ to $f'$, $\mathds 1_{e\in \eta}$ is the
indicator that $e$ is occupied by a dimer and $\sigma_e=+1$ or $-1$
according to whether $e$ is crossed with the white vertex on the right
or left. The r.h.s. of \eqref{eq:hgrad} is well-known to be independent of the choice of $C_{f\to f'}$.

In order for the shuffling algorithm to define a Markovian evolution
of the height profile, we have to complement the definition of the
maps $T_k$ with a prescription of how the height offset evolves as
time $k$ increases.  The convention that we adopt here is slightly
different from that of \cite{zhang2018domino,CT:18}.  We start with
the following observation, which is immediately verified from the
definition of $T_k$ and of the height function (recall that vertex
colors are swapped at each step).  Let $f,f'$ be any two faces that
are odd at time $k$, i.e. they have coordinates $(i,j)$ and $(i',j')$
respectively, with $i+j=k+1\mod 2$ and $i'+j'=k+1\mod 2$; then,
\[
h_{T_{k+1} (\eta)}(f)-h_{T_{k+1}(\eta)}(f')=h_{\eta}(f)-h_{\eta}(f').
  \]
Therefore, we make the following choice: 
\begin{Definition}
  \label{def:ho}
  If $f$ is an odd  face at time $k$, then
  \begin{eqnarray}
    \label{eq:ch1}
    h_{T_{k+1} (\eta)}(f)=h_{\eta}(f).
  \end{eqnarray}
\end{Definition}
This convention fixes unambiguously the whole height function of $\eta_{k+1}$ and in
particular the value of $h_{T_{k+1}(\eta)}(f)$ for \emph{even} faces $f$.  Namely, let $f$ be any face  and let $\eta|_{\partial f}$
(resp.  $T_{k+1}(\eta)|_{\partial f}$) be the restriction of the dimer
configuration $\eta$ (resp. $T_{k+1}(\eta)$) to the four boundary edges of
$f$. Then, one may check by direct inspection starting from the definition of $T_{k+1}$
that,  if $f$ is even at time $k$, then
\begin{eqnarray}
  \label{eq:ch2}
  h_{T_{k+1} (\eta)}(f)-h_{\eta}(f)= \frac{
H[\eta]+H[T_{k+1}(\eta)]-V[\eta]-V[T_{k+1}(\eta)]}4,
\end{eqnarray}
where (denoting $e_1,\dots e_4$ the four boundary edges of $f$, labeled clockwise from the top one),
\begin{eqnarray}
  \label{eq:hor}
  H[\eta]=\mathds 1_{e_1}(\eta)+ \mathds 1_{e_3}(\eta)
\end{eqnarray}
and
\begin{eqnarray}
  \label{eq:vert}
  V[\eta]=\mathds 1_{e_2}(\eta)+ \mathds 1_{e_4}(\eta); 
\end{eqnarray}
see Fig.~\ref{fig:spiderheights}.
\begin{figure}
	\begin{center}
		\includegraphics[height=2.5cm]{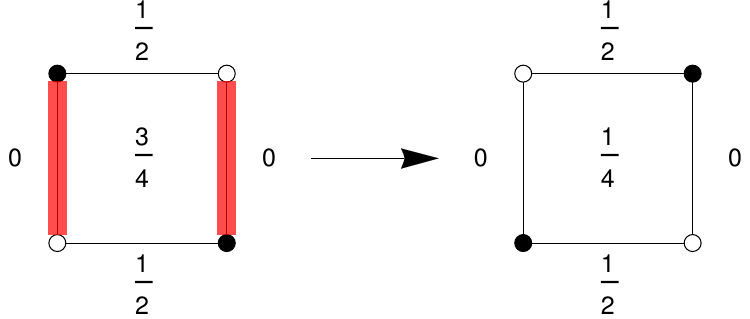}
		\hspace{15mm}
		\includegraphics[height=2.5cm]{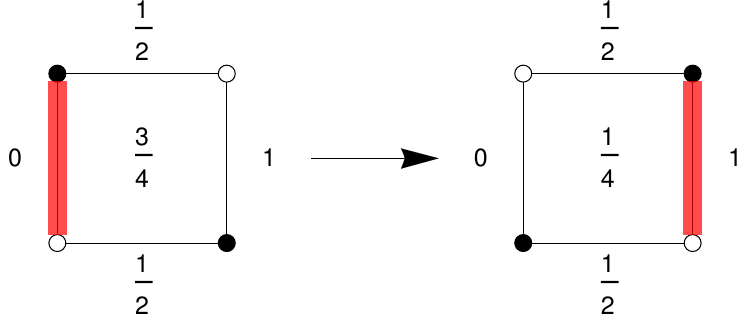}
		\includegraphics[height=2.5cm]{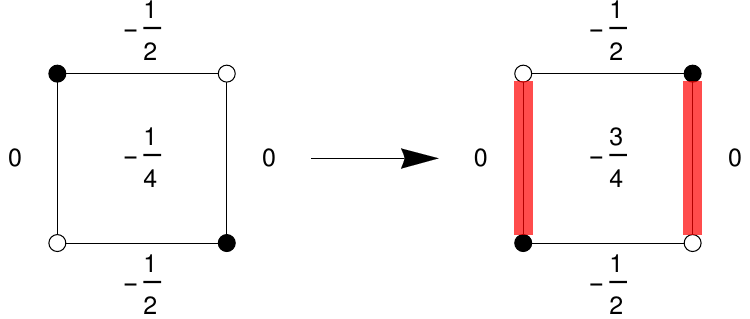}
		\caption{The height function change at even faces for
                  configurations only concerning vertical dimers. One
                  can easily obtain the same picture for horizontal
                  edges, by rotating each configuration by $\pi/4$,
                  interchanging the white and black vertices (and as a
                  result multiplying all heights by $-1$). }
		\label{fig:spiderheights}
	\end{center}
\end{figure}


\subsection{Periodic weights} 
\label{sec:pwat}

In this section, we introduce briefly some of the main aspects of the
dimer model machinery needed for the formulation of the main result.
Since we are interested in stochastic growth in a translationally
invariant situation, here we specialize to the case where the edge
weights are periodic in both directions of space.  Let the fundamental
domain $D_{0,0}$ of size $2n\times 2n, n\in\mathbb N$, consist of the
vertices $\{(i,j): 0 \leq i,j \leq 2n-1 
\}$, 
half of which are black and half white.  For $ 0 \leq j \leq 2n-1$,
the vertices $(2n,j)$ are identified with the vertices $(0,j)$ but are
on the fundamental domain $D_{1,0}$ (obtained from $D_{0,0}$ via a
horizontal translation by $2n$), while for $0 \leq i \leq 2n-1$, the
vertices $(i,2n)$ are identified with the vertices $(i,0)$ but on the
fundamental domain $D_{0,1}$.  The edge weights are chosen on all
edges on $D_{0,0}$ and its boundary edges and then extented by
periodicity to the whole graph.  Call this weighting
$\mathtt{w}_0$.%

Underlying the dimer model theory is the {characteristic polynomial}
$P$.  To define $P$, consider $D_{0,0}$ embedded on a $2n\times 2n$  torus as
above 
and let $\mathrm{wt}(x,y)$ denote the weight of the edge $(x,y)$ for
two vertices $x$ and $y$ of $D_{0,0}$.  Given $z,w\in \mathbb C$,
define $K(z,w)$ to be the Kasteleyn matrix with rows indexed by white vertices
and columns indexed by black vertices of $D_{0,0}$, with
\[
	(K(z,w))_{xy} = \left\{ \begin{array}{lll}
		\mathrm{wt}(x,y)z^a  & \text{if} & \mbox{$(x,y)$ is a horizontal edge,}\\
		\mathrm{i} \,\mathrm{wt}(x,y)  w^b & \text{if} &\mbox{$(x,y)$ is a vertical edge,}\\
		0 & \text{if} & \mbox{$(x,y)$ is not an edge}
	\end{array} \right.
	\]
where $x$ is a white vertex and $y$ is black vertex in $D_{0,0}$, and
\[
	a= \left\{ \begin{array}{ll}
	1 & \mbox{if $x=(2n-1,k)$ and $y=(0,k)$}\\
	-1 & \mbox{if $x=(0,k)$ and $y=(2n-1,k)$}\\
	0 & \mbox{otherwise} \end{array} \right. 
	\]
	and
	\[
	b= \left\{ \begin{array}{ll}
	1 & \mbox{if $x=(l,2n-1)$ and $y=(l,0)$}\\
	-1 & \mbox{if $x=(l,0)$ and $y=(l,2n-1)$}\\
	0 & \mbox{otherwise} \end{array}\right. 
\]
for $0 \leq k,l \leq 2n-1$.
The Laurent polynomial $P(z,w)=\det K(z,w)$ is called ``characteristic polynomial''~\cite{KOS03}.
Of course, $P$ depends on $n$ and on the weights.

From~\cite{KOS03}, the Newton Polygon (depending on $n$) is defined
to be
\[
	N(P)=\mbox{convex hull}\{ (j,k) \in \mathbb{Z}^2|z^j w^k \mbox{ is a monomial in }P(z,w) \}\subset \mathbb R^2.
	\]
One can check,  for the $K(z,w)$ specified above, that $N(P)$ is the (closed) square with vertices
$(\pm n,0),(0,\pm n)$. 

A probability measure $\mu$ on $\Omega$  is said to be an ergodic Gibbs measure
(corresponding to the edge weights $\mathtt w_0$) if:
\begin{itemize}
\item it  is invariant and ergodic
with respect to horizontal/vertical translations by multiples
of $2n$;
\item it satisfies the following  Dobrushin-Lanford-Ruelle (DLR) property. Given any finite subset
  $\Lambda$ of edges and any dimer configuration $\bar \eta\in \Omega$,
  let $\Omega_{\Lambda,\bar \eta}$ be the (finite) set of dimer configurations
  $\eta\in\Omega$ that coincide with $\bar\eta$ outside
  $\Lambda$. Then, conditionally on $\eta=\bar \eta$ outside
  $\Lambda$, the $\mu$-probability of a configuration $\eta$ is proportional
to  the product 
\[\prod_{e\in \eta\cap \Lambda} \mathtt w_0(e)\]
of $\mathtt w_0$-weights of the edges in $\Lambda $ occupied by dimers.

\end{itemize}
Thanks to translation invariance,  one may associate to
each ergodic Gibbs measure $\mu$ an average slope $\rho=(\rho_1,\rho_2)$. Here, $\rho_1$
(resp. $\rho_2$) is the expected height difference between a face in
$D_{0,0}$ and its translate in $D_{1,0}$ (resp. $D_{0,1}$). The slope
$\rho$ is contained in the  Newton polygon $N(P)$. Moreover, provided that  $\rho$ belongs $\stackrel{\circ}{N(P)}$, the
interior of $N(P)$, there exists a unique Gibbs measure with slope $\rho$
\cite{KOS03} and we denote it by
$\pi_{\rho,\mathtt{w}_0}$. This measure is known to be
determinantal, in the sense that the probability that $r$ given edges
$e_1,\dots,e_r$ belong to $\eta$ is given by the determinant of an
$r\times r$ matrix, whose entries are elements of the so-called
inverse Kasteleyn matrix.

 Define the Ronkin function associated to $P$ as
 \begin{eqnarray}
   \label{eq:Ronkin}
   R(B)= \frac{1}{(2\pi \mathrm{i})^2} \int \int_{\substack{ |z|=e^{B_1} \\ |w|=e^{B_2}}} \log |P(z,w)| \frac{dz}{z} \frac{dw}{w}   
 \end{eqnarray}
      for $B=(B_1,B_2)\in\mathbb R^2$.  From~\cite{KOS03}, $R$ is the
      Legendre transform of the so-called  surface tension $\sigma$ of the dimer
      model with the given periodic weights, i.e. 
      \begin{eqnarray}
        \label{eq:sigma}
	\sigma(\rho)=\sup_{B\in\mathbb R^2}(-R(B)+\rho\cdot B).        
      \end{eqnarray}
      We will recall later the relation between $\sigma$ and the
      ``limit shapes'' of the dimer model. 

      We write $ \stackrel\circ{N(P)}$, the interior of the Newton
      polygon, as the disjoint union of $\mathcal R$ (rough region)
      and $\mathcal S$ (smooth region), whose definition we recall
      now.  (Rough (resp. smooth) phases are called
      ``liquid'' (resp. ``gaseous'') phases in \cite{KOS03}.)
      From~\cite{KOS03}, it is known that if
      $\rho\in \stackrel\circ{N(P)}$, two cases can occur:
      \begin{itemize}
      \item 
        either the measure $\pi_{\rho,\mathtt w_0}$ is  \emph{rough}, meaning that
        height fluctuations of $h_\eta(f)-h_{\eta}(f')$ grow
        logarithmically w.r.t. the distance between the faces
        $f,f'$. More precisely,
        \begin{eqnarray}
          {\rm Var}_{\pi_{\rho,\mathtt w_0}}(h_\eta(f)-h_{\eta}(f'))\sim \frac1{\pi^2}\log |f-f'|
        \end{eqnarray}
        as the distance $|f-f'|$ between $f$ and $f'$ diverges. Moreover, the scaling limit of the height profile is a Gaussian Free Field \cite{KenLectures}. We call $\mathcal R$ the set of such ``rough slopes''; 
\item or the measure $\pi_{\rho,\mathtt w_0}$ is \emph{smooth},
        meaning that height flucutations of $h_\eta(f)-h_{\eta}(f')$
        have uniformly bounded variance. In this case, $\sigma$
        has a cone singularity (i.e. $\nabla\sigma$ is discontinuous) at this
		      value of $\rho$. The set of ``smooth slopes'' is denoted $\mathcal S$.       \end{itemize}
{In both cases, $\sigma(\cdot)$ is strictly convex.}

      From \cite{KOS03}, it is further known that $\mathcal S$ is a finite set and moreover
        \begin{eqnarray}
          \label{eq:rough}
\mathcal S\subset \Bigl[\stackrel\circ{N(P)}\cap \;\mathbb Z^2\Bigr];
        \end{eqnarray}       for
        generic edge weights $\mathtt w_0$,  $\mathcal S$ actually coincides with the whole  $\stackrel\circ{N(P)}\cap \;\mathbb Z^2$, which contains
        $2n(n-1)+1$ points. However, this may fail for particular
        choices of weights: for instance, when all edge-weights are
        equal, then it is known that $\mathcal S=\emptyset$.
        

\subsubsection{Shuffling algorithm with periodic weights}

A remarkable feature of
the shuffling algorithm is the following (see for instance
\cite[Proposition 3.1]{CT:18} and \cite[Prop. 2.2]{zhang2018domino}):
\begin{prop}
  \label{prop:misuramappata}
If the initial condition
$\eta_0$ at time $0$ is drawn from $\pi_{\rho,\mathtt w_0}$ (i.e., $\eta_0\sim \pi_{\rho,\mathtt{w}_0}$), then at time
$k$ one has $\eta_k\sim \pi_{\rho,\mathtt{w}_k}$.   
\end{prop}
 If we had not swapped vertex colors at each step, the slope
$\rho$ would swap to $-\rho$ at each step.  There are two observations
that we will need going forward. The first one is that the
characteristic polynomial only changes by a multiplicative constant factor when the weights $\mathtt w_k$ are replaced by
$\mathtt w_{k+1}$~\cite{GK13}
  In particular, in view of \eqref{eq:Ronkin} and
\eqref{eq:sigma}, this implies that the surface tension
$\sigma(\cdot)$ for weights $\mathtt w_k$ equals that for weights
$\mathtt w_{k+1}$, up to an additive constant.  Another consequence is
the following: since the rough or smooth nature of
$\pi_{\rho,\mathtt w}$ depends on $\mathtt w$ only through the zeros
of the characteristic polynomial $P(z,w)$ on the torus
$\{z,w\in \mathbb C: |z|=|w|=1\}$ \cite{KOS03}, we deduce that
$\pi_{\rho,\mathtt w_{k+1}}$ is rough (resp. smooth) iff
$\pi_{\rho,\mathtt w_{k}}$ is. In other words, the condition
$\rho\in \mathcal R$ does not depend on $k$.

Another important observation is the following: even though weights
$\mathtt w_0$ (and therefore $\mathtt w_k$) are periodic \emph{in
  space}, the sequence $\{\mathtt w_k\}_{k\ge0}$ is in general
\emph{not periodic} w.r.t. the time index $k$.  Time-periodicity can,
however, hold for special choices of $\mathtt w_0$ and indeed the
cases studied in \cite{CT:18,zhang2018domino} are time-periodic.

\subsection{The Aztec diamond}

\label{sec:Azteco}
The Aztec diamond $A_N$ of size $N$ is the subset of the graph
$\mathbb{Z}^2$ whose vertices  have Cartesian coordinates
$(x_1,x_2)$ satisfying the condition $|x_1-1/2|+|x_2-1/2| \leq N$.
We let  $E^+_N$
denote the set of edges outgoing from $A_N$, $F^+_N$ the set of faces
not in $A_N$ but neighboring $A_N$ and $F_N$ the set of internal faces of $A_N$.
\begin{figure}
	\begin{center}
		\includegraphics[height=6cm]{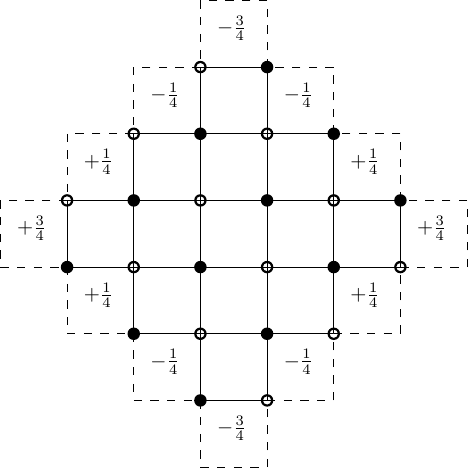}
		\caption{The Aztec diamond $A_3$ (full edges and colored vertices). The dashed edges are the boundary edges in $E^+_3$ and the faces containing dashed edges are the faces in $F^+_3$. The height function on $F^+_3$ is given.}
		\label{fig:Azteco}
	\end{center}
\end{figure}

Let
$\pi_{\mathtt w,N}$ be the probability measure on $\Omega_N$, the set
of perfect matchings of $A_N$, where the weight of a configuration is
proportional to the product of the $\mathtt w$-weights of edges occupied by
dimers. 
Since all vertices of $A_N$ are matched among themselves, all edges in
$E^+_N$ are empty and therefore the height difference between two
faces in $F^+_N$ is independent of the choice of $\eta\in\Omega_N$. We assume
that the coloring of the vertices is such that the vertex of
coordinates $(-N+1,1)$ is white. We fix the height offset as in Fig.
\ref{fig:Azteco}, by setting the height to $+N/4$ on the leftmost
face of $F^+_N$; then, the boundary height ranges from $-N/4$ to
$+N/4$.

The height function in $A_N$ satisfies a limit shape phenomenon (or
law of large numbers) as $N\to\infty$. Namely, rescale the lattice
mesh by $1/(2nN)$ and call $\hat A_N$ the rescaled Aztec diamond (and correspondingly denote $\hat E_N^+,\hat F_N^+,\hat F_N$ the analog of $E_N^+,F_N^+,F_N$).  The  union of the faces of $\hat A_N$ tends to the
square
\begin{eqnarray}
  \label{eq:Q}
Q=\{(x_1,x_2):|x_1|+|x_2|\le 1/(2n)\}.  
\end{eqnarray}
Define the rescaled height function $\hat h_\eta:\hat F_N\mapsto \mathbb R$ as
\[
\hat h_\eta(\hat f):=\frac1N h_\eta(f), 
\]
with $f\in F_N$ the face of $A_N$ that corresponds to $\hat f$ before
rescaling.  Thanks to the factor $2n$ in the rescaling, $\hat h_\eta$
is a Lipschitz function whose gradient is contained in the Newton
polygon $N(P)$. Note that, if $\hat f=\hat f_N$ is a face in $\hat F^+_N$ whose
center tends to $x=(x_1,x_2)\in \partial Q$ as $N\to\infty$, then
\begin{eqnarray}
  \label{eq:partialQ}
 \hat h_\eta(\hat f)\stackrel{N\to\infty}= \psi_{\partial Q}(x):=\frac n2\left({|x_1|-|x_2|}\right).  
\end{eqnarray}
 The limit shape theorem
(cf. \cite{CKP01} for the model with uniform weights and
\cite{kuchumov2017limit} for the general periodic case) states
that there exists a Lipschitz function $\psi_{\mathtt w}:Q\mapsto \mathbb R$
that coincides with $\psi_{\partial Q}$ on $\partial Q$, such that for
every $\delta>0$,
\begin{eqnarray}
  \label{eq:lln}
\lim_{N\to\infty}  \pi_{\mathtt w,N}\Bigl(\exists \hat f\in \hat F_N:| \hat h_\eta(\hat f)-\psi_{\mathtt w}(\hat f)|>\delta\Bigr)=0.
\end{eqnarray}
Here, with some abuse of notation, $\psi_{\mathtt w}(\hat f)$ means  $\psi_{\mathtt w}$
computed at the center of the face $\hat f$.
The ``limit shape'' $\psi_{\mathtt w}$ is characterized by being the unique minimizer of the
surface tension functional
\[
\int_Q \sigma(\nabla \psi) dx
\]
among Lipschitz functions that equal $\psi_{\partial Q}$ on the
boundary. While the boundary condition does not depend on $\mathtt w$,
the limit shape does (through the surface tension), but
$\psi_{\mathtt w_{k+1}}=\psi_{\mathtt w_{k}}$ because, as we already
mentioned, $\sigma$ changes only by an additive constant when
$\mathtt w_k$ is changed into $\mathtt w_{k+1}$.

\subsection{Statement of Main theorem} 
Our main result concerns the average speed of growth for the Markov
process in the infinite graph, started from
$\pi_{\rho,\mathtt{w}}$. By definition, this is given by the limit
(provided it exists)
\begin{eqnarray}
  \label{eq:prov}
  v_{\mathtt w}(\rho):=  \lim_{k\to\infty}\frac1k \sum_{j=1}^k\left(\mathbb E_{\pi_{\rho,\mathtt{w}}}(h_{\eta_j}(f))-\mathbb E_{\pi_{\rho,\mathtt{w}}}(h_{\eta_{j-1}}(f))\right)
\end{eqnarray}
with $f$ any face of $\mathbb Z^2$ and $\mathbb E_\nu$ the law of the
process started from the probability measure $\nu$. Note that every second term in the sum is zero because every second time the face $f$ is odd.

Since $\eta_j\sim \pi_{\rho,\mathtt w_j}$ with $\mathtt w_0\equiv \mathtt w$, each non-zero term in the
sum could in principle be computed via \eqref{eq:ch2} and Kasteleyn
theory, using the determinantal structure of the measure
$\pi_{\rho,\mathtt w_k}$.   Following this route, however, it is not
clear how to get any manageable expression or to prove that the limit
$k\to\infty$ in \eqref{eq:prov} even exists.  One reason is that, for generic periodic
weights, it is hard to invert the infinite-volume Kasteleyn matrix
explicitly. Fortunately, an alternative way exists, that leads to:
\begin{Theorem}
  \label{th:1}
	For every $\rho\in\stackrel\circ{N(P)}$ and positive periodic weighting $\mathtt w$, there exists $v=v_{\mathtt w}(\rho)$ such that, for any face $f\in \mathbb{Z}^2$,
  \begin{eqnarray}
    \label{eq:v}
    \lim_{k\to\infty}\frac1k\mathbb E_{\pi_{\rho,{\mathtt w}}} (h_{\eta_{k}}(f)-h_{\eta_{0}}(f))=v_{\mathtt w}(\rho). 
  \end{eqnarray}
  The speed $v_{\mathtt w}(\cdot)$ is determined as follows: let
  $\psi_{\mathtt w}(\cdot)$ be the limit shape for the dimer model in
  the Aztec diamond with weights $\mathtt w$ and let
  $x_{\mathtt w}(\rho)\in \stackrel\circ Q$ (the interior of the unit square in \eqref{eq:Q}) be a point such that
  $\nabla \psi_{\mathtt w}(x_{\mathtt w}(\rho))=\rho$. Then
  \begin{eqnarray}
\label{eq:chiev}
    v_{\mathtt w}(\rho)=\psi_{\mathtt w}(x_{\mathtt w}(\rho))-x_{\mathtt w}(\rho)\cdot \rho.
  \end{eqnarray}

  On the rough region $\mathcal R$, $v_{\mathtt w}(\cdot)$ is
  $C^\infty$and $\det(D^2 v_{\mathtt w})< 0$. On the other hand,
  $D v_{\mathtt w}$ is discontinuous at every $\rho\in \mathcal S$.
\end{Theorem}
A few comments are in order:
\begin{itemize}
\item the existence of $x_{\mathtt w}(\rho)$ is part of the statement. Uniqueness in general fails (the limit shape
$\psi_{\mathtt w}$ may have ``facets'', i.e. open regions where it is
affine) but for $\rho\in\mathcal R$, the point $x_{\mathtt w}(\rho)$ is unique (see Section \ref{sec:proprieta}).
\item Using smoothness of $v_{\mathtt w}(\cdot)$ on $\mathcal R$ and
  \eqref{eq:chiev}, one sees that
\begin{eqnarray}
  \label{eq:chieDv}
  D v_{\mathtt w}(\rho)=-x_{\mathtt w}(\rho).
\end{eqnarray}
Note that the r.h.s. of \eqref{eq:chiev} looks like (minus) the
Legendre transform of $\psi_{\mathtt w}$, except that there is no
infimum over $x$ and in fact neither $v_{\mathtt w}$ nor
$\psi_{\mathtt w}$ have any definite convexity.
\item It was observed in \cite{KO07} that the Euler-Lagrange equation
  satisfied by the limit shape $\psi_{\mathtt w}$ of dimer models can be written (in the ``rough region'' where the limit shape is $C^2$)
  in terms of a first-order PDE (``complex Burgers equation'') for a
  complex pair $(z,w)$ related by the relations $P(z,w)=0$ and
  $\pi\nabla \psi_{\mathtt w}=(-{\rm arg}(w),{\rm arg}(z))$. Locally, these
  relations give a bijection between $z$ and $\rho=\nabla\psi_{\mathtt w}$.  Then,
  using \cite[Sec. 3]{borodin2018two}, the above Theorem \ref{th:1}
  can be complemented by  the following  statement:
  \[\hat v_{\mathtt w}(z):=v_{\mathtt w}(\rho(z))\; \text{ is a harmonic
  function of } \; z.\]
\item For a special case of two-periodic weights ($n=1$), it was found via explicit computation in \cite[Th. 3.11]{CT:18} that the behavior of $v_{\mathtt w}$ near the unique gas slope $\rho=0$ is of the type
  \begin{eqnarray}
    v_{\mathtt w}(\rho)\stackrel{\rho\to0}=|\rho|f_1({\rm arg}(\rho))+|\rho|^3f_3({\rm arg}(\rho))+O(|\rho|^5).
  \end{eqnarray}
  The absence of the  $|\rho|^2$ term can be given an interesting interpretation. In fact, this is a simple consequence of  formula \eqref{eq:chiev} plus  the fact that, if $x$ approaches a point $x_0$ on the boundary of the ``facet'' where $\nabla\psi_{\mathtt w}\equiv 0$, then generically $\psi_{\mathtt w}(x)-\psi_{\mathtt w}(x_0)$ vanishes as $|x-x_0|^{3/2}$ \cite{KO07} (this behavior is referred to as ``Pokrovsky-Talapov law'' \cite{pokrovskii1980theory}).
\end{itemize}
\subsubsection{Fluctuations}
\label{sec:fluct}
 One can further prove that height fluctuations grow slowly (at most
 logarithmically) in time, as is typical for growth models in the AKPZ
 universality class. In fact, one has uniformly in $k\ge 1$
\begin{eqnarray}
\nonumber
  \mathbb P_{\pi_{\rho,\mathtt w}}(|h_{\eta_k}(f)-h_{\eta_0}(f) -\mathbb E_{\pi_{\rho,\mathtt w}} (h_{\eta_k}(f)-h_{\eta_0}(f))|\ge u g(k))\le \frac{c}{u^2}
\end{eqnarray}
for some constant $c$,
where $g(k)=\sqrt{\log(k+1)}$ if $\rho\in\mathcal R$ and $g(k)\equiv 1$ if $\rho\in\mathcal S$.  The proof of this fact works  the same as in \cite{CT:18}
 so we will not add details (the speed of convergence $O(u^{-2})$ was not explicitly stated in \cite{CT:18}, but it
 can be immediately extracted from the proof).
 Note in particular that
 \begin{equation}
   \label{eq:nip}
   \mathbb P_{\pi_{\rho,\mathtt w}}(|h_{\eta_k}(f)-h_{\eta_0}(f) -\mathbb E_{\pi_{\rho,\mathtt w}} (h_{\eta_k}(f)-h_{\eta_0}(f))|\ge \delta k)\le \frac{c [\log (k+1)]^2}{\delta^2k^2}
 \end{equation}
 and since the r.h.s. is summable in $k$, one can upgrade \eqref{eq:v} to
 the almost-sure convergence, with respect to the joint law of the initial condition and of the process,
 \begin{eqnarray}
   \label{eq:asc}
   \lim_{k\to\infty}\frac{h_{\eta_{k}}(f)-h_{\eta_{0}}(f)}k=v_{\mathtt w}(\rho). 
 \end{eqnarray}

\section{Identification of the speed of growth} \label{sec:Identification}
\label{sec:proofs}

In this section, we prove existence of the speed and  formula
\eqref{eq:chiev}.

\subsection{General properties of the dynamics}
\label{sec:Aztecodyn}

We need two general facts: the dynamics is monotone (it preserves
stochastic ordering among height profiles) and it is local
(information travels at most ballistically through the system).

Let us start with monotonicity. Given two dimer configurations
$\eta,\eta'$, we say that $h_\eta\preceq h_{\eta'}$ if
$h_\eta(f)\leq h_{\eta'}(f)$ for every face $f$. Given two initial
configurations $\eta_0,\eta'_0$, we can couple the two Markov chains
$\{\eta_k\}_{k\ge0}, \{\eta'_k\}_{k\ge0}$ in the following way
(\emph{global monotone coupling}): for any face $f$, if in both
configurations
$\eta_{k-1}|_{\partial f}=\eta'_{k-1}|_{\partial f}=\emptyset$, then
in the ``creation step'' of the shuffling map $T_k$ we choose the same
randomness to decide whether we add two vertical or two horizontal
dimers around $f$. Then, the following statement holds (it implies the
preservation of stochastic order mentioned above): if
$h_{\eta_0}\preceq h_{\eta'_0}$, then the same holds at all later
times $k$ \cite[Lemma 2.4]{zhang2018domino}.

As far as locality is concerned the point is that, by the definition
of the shuffling algorithm, the value of $h_{\eta_k}(f)$ is completely
determined by the height at time $k-1$ at the face $f$ and at its four
neighbors (this determines the dimer configuration $\eta_{k-1}$ on
$\partial f$), plus the randomness used to create parallel dimers at
$f$, if the face is even and $\eta_{k-1}|_{\partial f}=\emptyset
$. From this, it is immediate to deduce the following:
\begin{prop}
  \label{prop:loc}
  Let $\eta_0,\eta'_0$ be two dimer configurations whose height
  coincides on all faces at $\ell_1$-distance up to $N+1$ from a given
  face $f$. Couple the Markov chains started from $\eta_0,\eta'_0$ via
  the global monotone coupling. Then, $h_{\eta_k}(f)=h_{\eta'_k}(f)$
  for every $k\le N$.
\end{prop}

Let us also describe in some more detail how the shuffling algorithm
works on the Aztec diamond (this is the framework where the algorithm
was originally introduced \cite{EKLP92,Pro03}). In a step of the
algorithm, a dimer configuration $\eta$ on $A_N$ is mapped to a
configuration $\eta'$ on the larger domain $A_{N+1}$.  Suppose that we
have $\eta_{N}\in\Omega_{N}$, i.e. a dimer configuration on the
diamond of size $N$. We can also view $\eta_{N}$ as a subset of
edges of $A_{N+1}$ (but not a perfect matching, since the boundary
vertices are necessarily unmatched).  To construct $\eta_{N+1}$, apply the
map $T_{N+1}$ in $A_N$ to $\eta_{N}$ (with weights $\mathtt w_N$ as
above). Note that the faces in $A_{N+1}$ that are closest to the boundary,
i.e. the faces in $F^+_{N}$, are even. It is well known that the
resulting dimer configuration $\eta_{N+1}$ is a perfect matching of
$A_{N+1}$. Due to the swapping of colors, at the next step the faces in
$F^+_{N+1}$ are again even and the procedure goes on.  

The analog of Proposition \ref{prop:misuramappata} in the Aztec
diamond is the well known fact that, if we start at time zero with a
configuration $\eta_0$ on $A_N$ such that
$\eta_0\sim \pi_{\mathtt w_0,N}$ for certain periodic weights
$\mathtt w_0$, then at time $k$ one has
$\eta_k\sim \pi_{\mathtt w_{k},N+k}$.

There is an important point to be discussed: when we introduced the
shuffling algorithm on the infinite lattice, we fixed the evolution of
the height offset via Definition \ref{def:ho}. On the other hand, on
the Aztec diamond the height offset is fixed by the requirement that
the left-most face in $F^+_k$ has height $k/4$. These two conventions
must be compatible, i.e., if we adopt the convention \eqref{eq:ch2}
for the evolution of the height function, then the height on the
left-most face of $F^+_k$ must be $k/4$ deterministically. This is
easily seen inductively in $k$, as explained in the caption of Fig.
\ref{fig:compatibili}.
\begin{figure}
	\begin{center}
		\includegraphics[height=6cm]{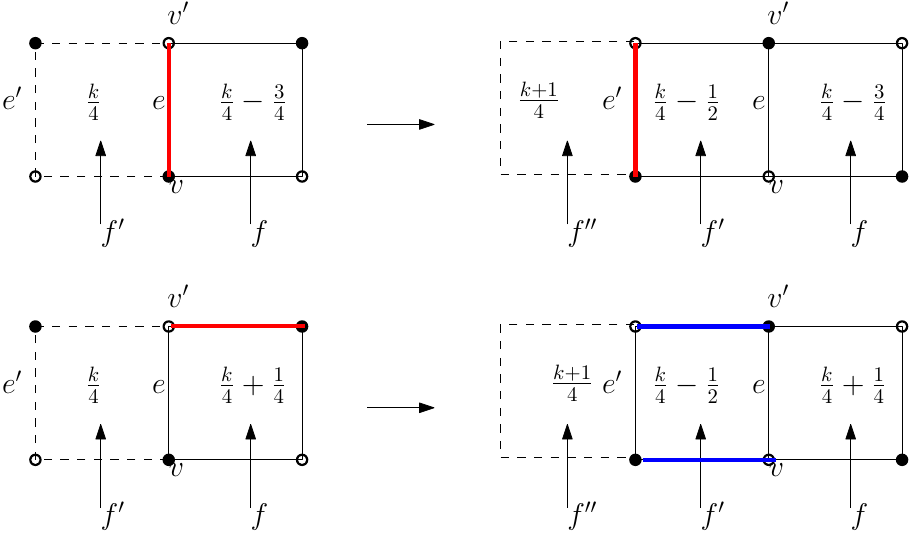}
		\caption{Let $f$ be the left-most face of $A_k$ and
                  $f'$ the left-most face of $F^+_k$, where 
                  (by inductive assumption) the height is equal to $k/4$. Suppose
                  (top drawing) that $v,v'$ are matched in
                  $\eta_k$. Then, in the application of $T_{k+1}$ the
                  red dimer slides to edge $e'$ and vertex colors are
                  swapped. Since $f$ is odd at time $k$, its height is
                  unchanged and as a consequence the height at $f''$
                  (the left-most face in $F^+_{k+1}$) is $(k+1)/4$ as
                  it should be. If instead $v'$ is not matched with $v$
                  (bottom drawing) then $\eta_k$ has no dimer on the
                  boundary of the even face $f'$. Then, in the
                  application of $T_{k+1}$, two parallel dimers
                  (horizontal and drawn in blue in the example of the picture) are
                  created at $f'$. Again, using that the height at the
                  odd face $f$ does not change, one sees that the
                  height at $f''$ is $(k+1)/4$. }
		\label{fig:compatibili}
	\end{center}
\end{figure}

\subsection{The speed of growth}
\label{sec:speedo}

Here we prove the following:
\begin{prop}
  Let $\rho\in \stackrel\circ{N(P)}$ and assume that there exists
  $x_{\mathtt w}(\rho)$ in the interior of $Q$, such that
  $\psi_{\mathtt w}(\cdot)$ is $C^1$ in a neighborhood of
  $x_{\mathtt w}(\rho)$ and
  $\nabla\psi_{\mathtt w}(x_{\mathtt w}(\rho))=\rho$. Then, the limit
  in \eqref{eq:v} exists and \eqref{eq:chiev} holds.
\end{prop}
 The existence of
$x_{\mathtt w}(\rho)$ for every $\rho$ in the interior of the Newton
polygon will be proved in the next section.
\begin{proof}
For an integer $N$, let $\bar f_N$ be a face of $A_N$ whose center is
at minimal distance from $(2nN) x_{\mathtt w}(\rho)$. One should think of
$N$ as being a large multiple of  $k$, the time in \eqref{eq:v}, with
$\epsilon:=k/N$ that will be sent to zero at the end. For later
convenience, we let $\Lambda_{N,\epsilon}$ be the square box of side
$2k+1$ centered at the face $\bar f_N$.  Recall that $A_N$ denotes the
$N\times N$ Aztec diamond and take the edge weights to be given by
$\mathtt w$.  We run the shuffling dynamics in the Aztec diamond,
starting at time zero with the domain $A_N$ and with an initial
condition $\tilde \eta_0$ sampled from $\pi_{\mathtt w,N}$. We denote
$\tilde \eta_k$ the configuration at time $k$, where the tilda is used
just to distinguish this from the evolution in the infinite graph. The
height function of $\tilde\eta_0$ is concentrated at the limit shape
$\psi_{\mathtt w}(\cdot)$. In particular, from \eqref{eq:lln} with
$\delta=\epsilon^2$ we have
\begin{eqnarray}
  \label{fatica}
  \pi_{\mathtt w,N} \left[\left|h_{\tilde \eta_0}(f)- N \psi_{\mathtt w}(f/(2nN))\right|\le N \epsilon^2 \text{ for every } f\in \Lambda_{N,\epsilon}\right]\stackrel{N\to \infty}\to1.
\end{eqnarray}
As before, we identify with some abuse of notation a
face $f$ with the point at its center.  As observed in Section
\ref{sec:Aztecodyn}, at time $k$, the configuration $\tilde \eta_k$ has
law $\pi_{\mathtt w_k,N+k}$ and we still have \eqref{fatica} with $N$
replaced by $N+k$.  Altogether, we see that
\begin{equation}
	\begin{split}
		&\mathbb E_{\pi_{\mathtt w,N}}\left[ \frac{ h_{\tilde \eta_k}(\bar f_N)-h_{\tilde \eta_0}(\bar f_N)}k\right]\\
		&=\frac Nk\left[(1+\epsilon)\psi_{\mathtt w}\left(\frac{x_{\mathtt w}(\rho)}{1+\epsilon}\right) - \psi_{\mathtt w}(x_{\mathtt w}(\rho))\right]+O\left( \epsilon^2 \frac Nk\right)\\
  \label{eq:perA}
  &= \psi_{\mathtt w}(x_{\mathtt w}(\rho))-x_{\mathtt w}(\rho)\cdot \rho+o_\epsilon(1)
	\end{split}
\end{equation}
where we used that $\nabla\psi_{\mathtt w}(x_{\mathtt w}(\rho))=\rho$
and the error term $o_\epsilon(1)$ vanishes as $\epsilon\to0$, since
the limit shape is $C^1$ around $x_{\mathtt w}(\rho)$. We also used
the fact that $|h_{\tilde \eta_0}|/N, |h_{\tilde \eta_k}|/N$ are
uniformly bounded for $k\le \epsilon N$, to deduce from \eqref{fatica}
a statement about their average.

Our goal now is to prove a statement analogous to \eqref{eq:perA} for
the dynamics $\{\eta_k\}_{k\ge0}$ on the infinite graph.  By
Proposition \ref{prop:loc}, the evolution of $h_{\eta_j}(\bar f_N),j\le k$ is
not influenced by the height function of $\eta_0$ outside $\Lambda_{N,\epsilon}$.
Recall that $\eta_0$ is sampled from the infinite-volume measure $\pi_{\rho,\mathtt w}$. Under
this probability measure, the height function is essentially linear, with slope $\rho$
	and sub-linear fluctuations.  More precisely, 
\begin{eqnarray}
  \label{eq:piatto}
  \pi_{\rho,\mathtt w}\left[\left|h_{\eta_0}(f)-h_{\eta_0}(\bar f_N)-\frac1{2n}\rho\cdot(f-\bar f_N)\right|\le N\epsilon^2 \;\forall\;f\in \Lambda_{N,\epsilon}\right]\stackrel{N\to\infty}\to1
\end{eqnarray}
where once more we have identified a face with its center and the
factor $1/(2n)$ is there because $\rho$ is the average height change
per fundamental domain. To get \eqref{eq:piatto}, observe first that \[\pi_{\rho,\mathtt w}[h_{\eta_0}(f)-h_{\eta_0}(\bar f_N)]=\frac1{2n}\rho\cdot(f-\bar f_N)+O(1),\] uniformly in $f\in\Lambda_{\epsilon,N}$ (the error term is there because $f$ is not necessarily an exact translation of $\bar f_N$ in a different fundamental domain). Also,   recall  that the fourth centered moment of $h_{\eta_0}(\bar f_N)-h_{\eta_0}(f)$ under $ \pi_{\rho,\mathtt w}$ grows at most like $(\log|\bar f_N-f|)^{2}$ for $|\bar f_N-f|$ large (see \cite[Section 4]{KOS03} for a $O(\log|\bar f_N-f|)$ bound on the variance; higher moments are treated analogously). Then,  a union bound over $f\in \Lambda_{N,\epsilon}$ and an application of Chebyshev's inequality leads to  \eqref{eq:piatto}.

Note that we have not yet specified the height offset
$h_{\eta_0}(\bar f_N)$ at time zero. We will fix it in such a way that,
with high probability  (w.h.p.) as $N\to\infty$,
\begin{eqnarray}
  \label{eq:then}
h_{\eta_0}(f)\le h_{\tilde \eta_0}(f) \quad \text{for every} \quad f\in \Lambda_{N,\epsilon}.  
\end{eqnarray}
  For this, note that \eqref{eq:piatto} implies that w.h.p.
  \begin{eqnarray}
    h_{\eta_0}(f)\le N\epsilon^2+h_{\eta_0}(\bar f_N)+\frac1{2n}\rho\cdot(f-\bar f_N) \quad \text{for every} \quad f\in \Lambda_{N,\epsilon}
  \end{eqnarray}
  while \eqref{fatica} and $C^1$ continuity of the limit shape implies that w.h.p.
  \begin{eqnarray}
    h_{\tilde \eta_0}(f)\ge N\psi_{\mathtt w}(x_{\mathtt w}(\rho))+\frac1{2n}\rho\cdot(f-\bar f_N)+R_{N,\epsilon}
  \end{eqnarray}
with $R_{N,\epsilon}/N\epsilon= o_\epsilon(1)$.
  Then, \eqref{eq:then} holds provided we choose
  \[
    h_{\eta_0}(\bar f_N)=N\psi_{\mathtt w}(x_{\mathtt w}(\rho))-N\epsilon^2-|R_{N,\epsilon}|.
  \]
  By monotonicity of the dynamics and Proposition \ref{prop:loc} we see that $h_{\eta_k}(\bar f_N)\le h_{\tilde\eta_k}(\bar f_N)$ and therefore, w.h.p.,
  \begin{eqnarray}
    \frac{h_{\eta_k}(\bar f_N)-h_{\eta_0}(\bar f_N)}k\le \frac1k\left(h_{\tilde\eta_k}(\bar f_N)-N\psi_{\mathtt w}(x_{\mathtt w}(\rho))+N\epsilon^2+|R_{N,\epsilon}|\right)
    \\
    \le \frac{h_{\tilde\eta_k}(\bar f_N)- h_{\tilde\eta_0}(\bar f_N)}k+o_\epsilon(1)
  \end{eqnarray}
  where we used \eqref{fatica} in the last step and
  $k=\epsilon N$. Note that $[h_{\eta_k}(f)-h_{\eta_0}(f)]/k$ is deterministically bounded by 
  $1$, so we can turn the statement w.h.p. into a
  statement in average and obtain that
  \begin{eqnarray}
    \label{quasi}
    \limsup_{k\to\infty}
    \mathbb E_{\pi_{\rho,\mathtt w}}\frac{h_{\eta_k}(\bar f_N)-h_{\eta_0}(\bar f_N)}k\le\psi_{\mathtt w}(x_{\mathtt w}(\rho))-x_{\mathtt w}(\rho)\cdot \rho+o_\epsilon(1)
  \end{eqnarray}
  where we used also \eqref{eq:perA}.  Note that the face $\bar f_N$
  depends on the time $k=N\epsilon$.  However, since the measure
  $\pi_{\rho,\mathtt w}$ is invariant by translations of multiples of
  $2n$ and the height function has bounded Lipschitz constant, we have
  \eqref{quasi} also for any fixed face $f$. Finally, we let
  $\epsilon\to0$.

  A lower bound is proven in the very same way and altogether the statements
  \eqref{eq:v} and \eqref{eq:chiev} follow.
\end{proof}

With similar arguments, we also obtain the following result, that will be useful later:
\begin{prop}
    \label{prop:damettere}
    If there exists $x$ in the interior of $Q$ such that
    $\psi_{\mathtt w}$ is $C^1$ in a neighborhood of $x$ and
	$\nabla\psi_{\mathtt w}(x)=\rho$ with $\rho=(\rho_1,\rho_2)$ at one of the four
    corners of the Newton polygon (i.e., $\rho=(\pm n,0)$ or
    $\rho=(0,\pm n)$) then
    \begin{eqnarray}
      \label{eq:ahah}
      \psi_{\mathtt w}(x)=\rho\cdot x+\frac1{4n}(|\rho_2|-|\rho_1|).
    \end{eqnarray}
  \end{prop}

  \begin{proof}
    Assume to fix ideas that $\rho=(n,0)$. As above, let $\bar f_N$
    be the face of $A_N$ closest to $(2n N) x$, let $k=N\epsilon$ and
    $\Lambda_{N,\epsilon}$ be the square of side $2k+1$ centered
    around $\bar f_N$. One has, in analogy with \eqref{eq:perA} and with the same argument,
	  \begin{eqnarray}\label{eq:expand}
    \frac1{|\Lambda_{N,\epsilon}|}\sum_{f\in\Lambda_{N,\epsilon}}  \mathbb E_{\pi_{\mathtt w,N}}\left[ \frac{ h_{\tilde \eta_k}(f)-h_{\tilde \eta_0}(f)}k\right]=  \psi_{\mathtt w}(x)-x\cdot \rho+o_\epsilon(1).
    \end{eqnarray}

    On the other hand, let $F$ be the collection of faces in
    $\Lambda_{N,\epsilon}$ (there are approximately $4 k^2$ of them).
    Write $F=F^{(+,j)}\cup F^{(-,j)}$, where $F^{(+,j)}$ contains the
    faces that are even at time $j$ and $F^{(-,j)}$ all the others.
    Because of \eqref{eq:lln} and the fact that
    $\nabla\psiw=(n,0)+o_\epsilon(1)$ in an $\epsilon$-neighborhood of
    $x$, from the definition of height function we see that, with
    probability $1-o_\epsilon(1)$, a proportion $1-o_\epsilon(1)$ of
    the dimers of $\tilde \eta_0$ in $\Lambda_{N,\epsilon}$ occupy a
    vertical edge with bottom white vertex. The same holds for
    $\tilde \eta_j,j\le k$, because $\tilde \eta_j$ has the same limit
    shape as $\tilde \eta_0$.  Therefore, a proportion
    $1-o_\epsilon(1)$ of the faces in $F^{(+,j)}$ have a single
    vertical dimer of $\tilde \eta_j$ along their boundary.  From
    \eqref{eq:ch2} we see that each such even face contributes $-1/2$
    to the height change from time $j$ to $j+1$. Since
    $|F^{(+,j)}|/|F|=1/2+o_\epsilon(1)$, the
    l.h.s. of~\eqref{eq:expand} equals also $-1/4+o_\epsilon(1)$ and
    \eqref{eq:ahah} follows.
  \end{proof}

\subsection{The limit shape}
Here we give some analytic properties of the limit shape $\psi_{\mathtt w}(\cdot)$ and prove the existence of $x_{\mathtt w}(\rho)$:
\begin{Theorem}
  \label{prop:KO}
  There exists a non-empty,  open subset $\mathcal F$ of the
  rescaled Aztec diamond $Q$ (cf. \eqref{eq:Q}) where
  $\psi_{\mathtt w}(\cdot)$ is $C^1$ and the gradient
  $\nabla \psi_{\mathtt w}(\cdot)\in \stackrel \circ{ N(P)}$. 
  For every
  $\rho\in \stackrel\circ{N(P)}$, there exists
  $x_{\mathtt w}(\rho)\in \mathcal F$ such that
 \[\nabla \psi_{\mathtt w}(x_{\mathtt w}(\rho))=\rho.\]
\end{Theorem}


\begin{proof}[Proof of Theorem \ref{prop:KO}]  
  For $x=(x_1,x_2)\in Q$, let
  \begin{multline}
    \label{eq:minmax}
    \psi^-(x) =\max[\phi_W(x),\phi_E(x)]
    :=\max\left[-n x_1-\frac14, n{x_1}-\frac14\right]
    =n{|x_1|}-\frac14,\\
    \psi^+(x) =\min[\phi_S(x),\phi_N(x)]
    :=\min\left[n{x_2}+\frac14,-n{x_2}+\frac14\right]
    =-n{|x_2|}+\frac14
  \end{multline}
  and note that $\psi^-$ (resp. $\psi^+$) is the mimimal
  (resp. maximal) Lipschitz function with gradient in $N(P)$ that equals $\psi_{\partial Q}$
  on $\partial Q.$   
  We let 
  \begin{eqnarray}
    \label{eq:mathcalF}
    \mathcal F_0:=\{x\in Q: \psi_{\mathtt w}(x)\ne \psi^\pm(x)\}\subset \stackrel \circ Q.
  \end{eqnarray}
  It is easy to see  the following (the proof is given below):
  \begin{Lemma}
    \label{lemma:pro}
The set  $\mathcal F_0$ is non-empty.
\end{Lemma}
We need some regularity properties of the limit shape
$\psi_{\mathtt w}$, and for this we appeal to \cite{SS10,Duse}.  Let
us compactify the Newton polygon by introducing a continuous map
$H:N(P)\mapsto S^2$ (the two-dimensional sphere) in such a way that
$\partial N(P)$ is mapped to a point of $S^2$ while $H$ is
a homeomorphism between $\stackrel\circ{N(P)}$ and
$H(\stackrel\circ{N(P)})$.  Then one has:
\begin{prop}
\cite[Th. 4.1 and Th. 1.3]{SS10}
  \label{prop:compactify}
The map
$x\mapsto H(\nabla\psi_{\mathtt w}(x))$ is continuous in the interior
of $Q$.   Moreover,
$\psi_{\mathtt w}$ is
$C^1$ in  $\mathcal F_0$.
\end{prop}

Define further the open set
  \begin{eqnarray}
    \label{eq:F0}
    \mathcal F:
    =\{x\in \mathcal F_0:\nabla\psi_{\mathtt w}(x)\in \stackrel \circ{N(P)}\},
  \end{eqnarray}
  that is the one appearing in the statement of 
Theorem \ref{prop:KO}.
 Decompose $\mathcal F$ as the union of the open set
  \begin{eqnarray}
    \label{eq:Fr}
    \mathcal F_R:=\{x\in  \mathcal F: \nabla\psi_{\mathtt w}(x)\in \mathcal R\}
  \end{eqnarray}
  and the closed set
  \begin{eqnarray}
    \label{eq:Fs}
    \mathcal F_S:=\{x\in \mathcal F: \nabla\psi_{\mathtt w}(x)\in \mathcal S\}.
  \end{eqnarray}
\begin{prop}
    \label{prop:notempty}
  The set $\mathcal F_R$ is non-empty.    
\end{prop}
Let us assume for the moment  Proposition \ref{prop:notempty}   (the proof is given below) and let us proceed with the proof of Theorem \ref{prop:KO}.
In general, $ \mathcal F_S$ consists of a collection of disjoint,
simply connected sets (these were called ``bubbles'' in \cite{KO07});
on each bubble, the gradient $\nabla\psi_{\mathtt w}$ is constant and
belongs to one of the finitely many slopes in $\mathcal S$.  It is
also known \cite{morrey2009multiple} that, on $ \mathcal F_R$, the
limit shape $\psi_{\mathtt w}$ is not just $C^1$ but actually
$C^\infty$, since the surface tension $\sigma(\rho)$ is $C^\infty$ for
$\rho\in \mathcal R$.  Therefore, in particular, the map
$D:x\mapsto \nabla \psi_{\mathtt w}(x)$ is a $C^1$ map from
$\mathcal F_R$ to $\mathcal R$.  The next step requires the
following:
  \begin{Theorem}\cite{Duse}
    \label{prop:clearly}
The map $D:x\mapsto \nabla \psi_{\mathtt w}(x)$ is  a proper map\footnote{\label{foot:natural} Properness does not hold for general domains
    $Q$ and boundary values $\psi_{\partial Q}$. Theorem
    \ref{prop:clearly} holds for the Aztec diamond because in this
    case $Q$ is a convex polygonal domain with sides perpendicular to the
    sides of the Newton polygon, and $\psi_{\partial Q}$ in
    \eqref{eq:partialQ} is a ``natural boundary value'' for $Q$. The
    notion of ``natural boundary value'' is defined in~\cite{Duse} and
    it requires in particular that, if the side $\ell$ of $Q$ is
    perpendicular to the side $[p_i,p_{i+1}]$ of the Newton polygon
    with $p_i,p_{i+1}$ two of its adjacent corners, then the
    derivative of $\psi_{\partial Q}$ along $\ell$ equals
    $\langle t_\ell,p_i\rangle$ with $t_\ell$ the tangent vector to
    $\partial Q$ along $\ell$. } from
  $\mathcal F_R$ to $\mathcal R$ (i.e. the pre-image of
  every compact subset of $\mathcal R$ is compact). 
\end{Theorem}
  Let us prove   that the Jacobian $\det(J(x))$ of the map $D $ is everywhere
  non-positive  on $\mathcal F_R$ and not identically zero.
  The Jacobian matrix equals 
\begin{eqnarray}
  \label{eq:J}
  J(x)=
  \begin{bmatrix}
    \partial^2_{x_1}\psi_{\mathtt w}(x)& \partial^2_{x_1 x_2}\psi_{\mathtt w}(x)\\
    \partial^2_{x_1 x_2}\psi_{\mathtt w}(x)& \partial^2_{x_2}\psi_{\mathtt w}(x)
  \end{bmatrix}.
\end{eqnarray}
On the other hand, on $\mathcal F_R$, $\psi_{\mathtt w}(\cdot)$
satisfies the Euler-Lagrange equation
\begin{eqnarray}
  \label{eq:EL}
  \sigma_{11}\partial^2_{x_1}\psi_{\mathtt w}(x)+2  \sigma_{12}  \partial^2_{x_1x_2}\psi_{\mathtt w}(x)+
   \sigma_{22} \partial^2_{x_2}\psi_{\mathtt w}(x)=0,
\end{eqnarray}
with $\sigma_{ab}$ the derivative of $\sigma(\rho)$ w.r.t. the
arguments $\rho_a,\rho_b$, computed at
$\rho:=\nabla\psi_{\mathtt w}(x)\in\mathcal R$.  For $\rho\in \mathcal R$, the
matrix $\{\sigma_{ab}\}_{a,b=1,2}$ is strictly positive definite, in
particular $|\sigma_{12}|<\sqrt{\sigma_{11}\sigma_{22}}$.  From this, we deduce that
\begin{eqnarray}
  \label{eq:contorto}
  \det(J(x))\ge0\Rightarrow J(x)_{i,j}=0 \quad \text{for every} \quad 1\le i,j\le 2.
\end{eqnarray} In fact, assume first that $\partial^2_{x_1x_2}\psi_{\mathtt w}(x)=0$. Then, $\partial^2_{x_1}\psi_{\mathtt w}(x) \partial^2_{x_2}\psi_{\mathtt
  w}(x)\ge0$ (because $\det J(x)\ge0$) but on the other hand \eqref{eq:EL} reduces to 
\begin{eqnarray}
  \label{eq:ELbis}
  \sigma_{11}\partial^2_{x_1}\psi_{\mathtt w}(x)+
  \sigma_{22} \partial^2_{x_2}\psi_{\mathtt w}(x)=0.
\end{eqnarray}
Since both $\sigma_{11},\sigma_{22}$ are strictly positive, the only
possibility is that
$\partial^2_{x_1}\psi_{\mathtt w}(x)= \partial^2_{x_2}\psi_{\mathtt
  w}(x)=0$.  On the other hand, assume (by contradiction) that
$\partial^2_{x_1x_2}\psi_{\mathtt w}(x)\ne0$, so that
$\partial^2_{x_1}\psi_{\mathtt w}(x) \partial^2_{x_2}\psi_{\mathtt
  w}(x)>0$. Then
\begin{eqnarray}
  \label{eq:contr}
  0\ge \sigma_{11}\partial^2_{x_1}\psi_{\mathtt w}(x)+\sigma_{22}\partial^2_{x_2}\psi_{\mathtt w}(x)-2 |\sigma_{12}|\sqrt{\partial^2_{x_1}\psi_{\mathtt w}(x)\partial^2_{x_2}\psi_{\mathtt w}(x)}\\
  > \sigma_{11}\partial^2_{x_1}\psi_{\mathtt w}(x)+\sigma_{22}\partial^2_{x_2}\psi_{\mathtt w}(x)-2 \sqrt{\sigma_{11}\sigma_{22}}\sqrt{\partial^2_{x_1}\psi_{\mathtt w}(x)\partial^2_{x_2}\psi_{\mathtt w}(x)}\ge0
\end{eqnarray}
which is a contradiction because the second inequality is
strict. Altogether, \eqref{eq:contorto} follows.  From this, we see
that $\det(J(\cdot))$ can vanish identically on $\mathcal F_R$ only if
$\psi_{\mathtt w}(\cdot)$ is affine, which is clearly not possible
in view of Proposition \ref{prop:clearly}.

We have that the map $D$ is proper and its Jacobian is non-negative
and not identically vanishing.  Then, by
\cite[Th. 1]{nijenhuis1962theorem}, we deduce that the map $D$ is
onto: for every $\rho\in \mathcal R$, there exists
$x_{\mathtt w}(\rho)\in \mathcal F_R$ with
$\nabla \psi_{\mathtt w}(x_{\mathtt w}(\rho))=\rho$.


It remains to show the existence of $x_{\mathtt w}(\rho)$ for every
$\rho\in \mathcal S$. Let $\{\rho_i\}$ be a sequence of slopes in
$\mathcal R$ that converges to $\rho$. Any limit point $\bar x$ of
$x_{\mathtt w}(\rho_i)$ is in $\mathcal F_0$ (because of Proposition \ref{prop:clearly}). Due to Proposition
\ref{prop:compactify}, the slope of $\psiw$ at $\bar x$ is
$\rho$, so we can set $x_{\mathtt w}(\rho):=\bar x$.
\end{proof}
We conclude this section by proving the two technical results, Lemma \ref{lemma:pro} and Proposition \ref{prop:notempty} that were stated above.

\begin{proof} [Proof of Lemma \ref{lemma:pro}]
Since   $\psi^-(x)<\psi^+(x)$ for every $x$ in the interior of $Q$ and $\psi_{\mathtt w}$ is continuous, we have
  just to exclude that $\psi_{\mathtt w}\equiv \psi^-$ or
  $\psi_{\mathtt w}\equiv \psi^+$.  Assume for instance that
  $\psi_{\mathtt w}\equiv \psi^+$; we are going to exhibit a function
  $\psi$, with the right boundary value, such that
\begin{eqnarray}
  \label{eq:atr}
  \int_Q \sigma(\nabla\psi)dx<\int_Q\sigma(\nabla\psi_{\mathtt w})d x=\frac{|Q|}2(\sigma(0,n)+\sigma(0,- n)).
\end{eqnarray}
For this purpose, let for $\epsilon>0$ small
\[
\psi(x):=\min(\psi^+(x),4 \epsilon n^2 x_1^2+(1/4-\epsilon)).
\]
It is immediate to see that $\psi(x)=4\epsilon n^2 x_1^2+(1/4-\epsilon)$ in
	\[S_\epsilon:=\{x:|x_2|\le \frac{\epsilon}{n} (1-4n^2 x_1^2)\}
\]
and $\psi(x)=\psi^+(x)$ in $Q\setminus S_\epsilon$, so in particular $\psi$ equals $\psi_{\partial Q}$  on $\partial Q$.
The difference between the r.h.s. and the l.h.s. of \eqref{eq:atr} is then
\begin{eqnarray}
  \label{eq:diff}
\int_{S_\epsilon}\left[ \frac12(\sigma(0,n)+\sigma(0,-n))-\sigma(8 \epsilon n^2 x_1,0)\right]dx.
\end{eqnarray}
Since $\sigma(\cdot)$ is strictly convex, one has
\[\frac12[\sigma(0,n)+\sigma(0,-n)]>\sigma(0,0).\]
Therefore, using continuity of $\sigma(\cdot)$, for $\epsilon$ small
enough the difference \eqref{eq:diff} is strictly positive and, as a
consequence, the minimizer $\psi_{\mathtt w}$ of the surface tension
functional cannot coincide with $\psi^+$.
\end{proof}

\begin{proof}[Proof of Propoposition \ref{prop:notempty}]
  We begin by making an observation on the shape of $\mathcal F_0$.
    Recall from \eqref{eq:minmax} the definition of $\phi_{a},a\in \{N,E, S, W\}$ and define the (possibly empty) regions
\begin{eqnarray}
    \label{eq:regcard}
    Q_{a}:=\{x\in \stackrel \circ Q: \psi_{\mathtt w}(x)=\phi_a(x)\}, \quad a\in \{N,E,S,W\}.
\end{eqnarray}
$Q_N$ belongs to the triangle $\{x\in Q: x_2\ge0\}$, otherwise
$\psi_{\mathtt w}$ would exceed the maximal function $\psi^+$; similar
statements hold for $Q_S,Q_E,Q_W$.  See Fig. \ref{fig:F}.
Also, it follows from \cite[Th. 4.2]{SS10} 
that $\partial Q_N\cap \stackrel\circ Q$ is the graph of a concave
function; analogously, $\partial Q_a\cap \stackrel\circ Q$ for
$a\in\{E,S,W\}$ is the graph of a concave function in a reference frame
rotated clockwise by $\pi/4,\pi/2$ and $3\pi/4$ respectively.  
\begin{figure}
	\begin{center}
		\includegraphics[height=6cm]{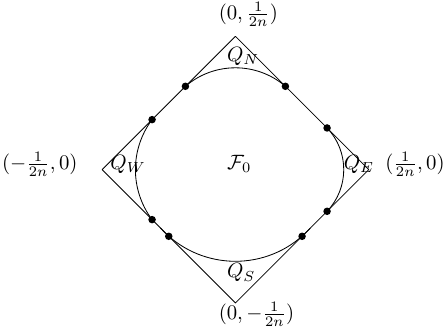}
		\caption{The square $Q$ with the convex region $\mathcal F_0$ and the four frozen regions $Q_a,a=N,W,S,E$.
               }
		\label{fig:F}
	\end{center}
\end{figure}
Because of the definition of $\psi^\pm$, we see that
\[
\mathcal F_0=\stackrel\circ Q\setminus \cup_{a\in\{N,E,S,W\}}Q_a.
  \]
Note that $\mathcal F_0$ is convex.


Before proving that $\mathcal F_R$ is non-empty, let us show that
$\mathcal F$ is non-empty. Let $\rho^{(a)},a\in\{N,E,S,W\}$ be the the
gradient of $\phi_a(\cdot)$ (these are also the four corners of
$N(P)$) and $\ell^{(i)},i\in\{NE,SE,SW,NW\}$ the open segment 
connecting $\rho^{(N)}$ to $\rho^{(E)}$ etc. 
  Remark  that if $x\in \mathcal F_0$, then
$\nabla\psi_{\mathtt w}(x)$ cannot coincide with any of the slopes
$\rho^{(a)}, a\in\{N,E,S,W\}$. In fact, thanks to Proposition
\ref{prop:damettere}, in this case one would have
  \[
\psi_{\mathtt w}(x)=\rho^{(a)}\cdot x+\frac1{4n}(|\rho^{(a)}_2|-|\rho^{(a)}_1|)=\phi_a(x),
\]
which contradicts the fact that $x\in\mathcal F_0$.
    Therefore,
    we have that
    \begin{eqnarray}
      \label{eq:FF0}
      \mathcal F=\mathcal F_0\setminus \cup_{i\in\{NE,SE,SW,NW\}} \mathcal F^{(i)},
    \end{eqnarray}
    with
  \[
    \mathcal F^{(i)}=\{x\in \mathcal F_0:\nabla\psi_{\mathtt w}(x)\in \ell^{(i)}\}.
  \]
  In general, the region $\mathcal F$ is a proper subset of $\mathcal F_0$, see Fig. \ref{fig:semifrozen}.
      \begin{figure}
	\begin{center}
		\includegraphics[height=6cm]{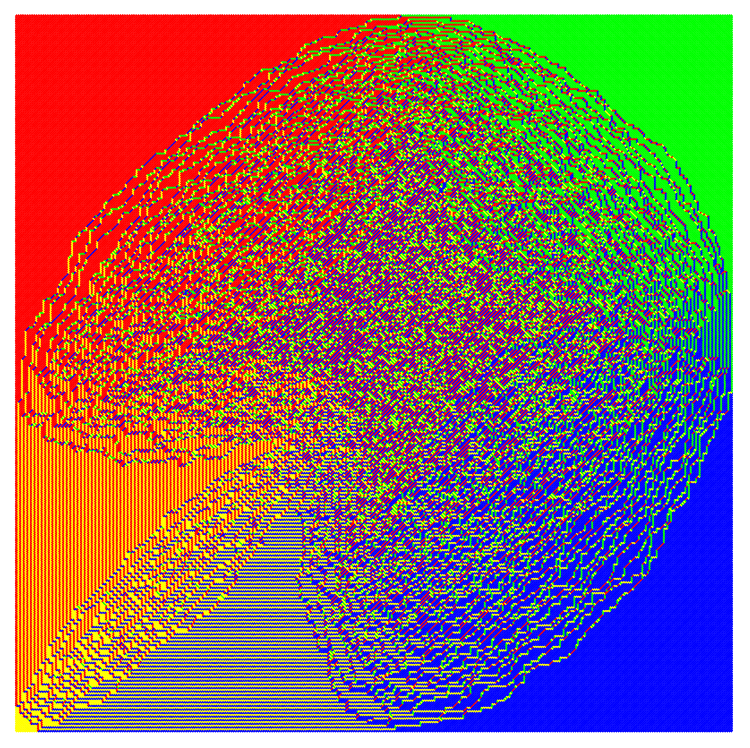}
		\caption{A random domino tiling of the Aztec diamond
                  of size $N=800$ (rotated by $45$ degrees) with edge
                  weights of period $n=2$ (the weights were randomly chosen on the fundamental domain and then extended by periodicity). The configuration is
                  obtained via the shuffling algorithm and it is
                  therefore a perfect sample from $\pi_{\mathtt
                    w,N}$. In addition to the frozen regions
                  $Q_N,Q_E,Q_S,Q_W$ adjacent to the corners of the
                  domain, where the gradient of the limit shape
                  $\psi_{\mathtt w}$ equals $(\pm n,0),(0,\pm n)$, one
                  remarks the presence of regions, adjacent to the
                  sides, where $\nabla\psi_{\mathtt w}$ belongs to  $\partial N(P)\setminus\{(\pm n,0),(0,\pm
                  n)\}$. These regions belong to $\mathcal F_0$ but
                  not to $\mathcal F$.  }
		\label{fig:semifrozen}
	\end{center}
\end{figure}

Using also the second statement in Proposition \ref{prop:compactify},
we conclude that if (by contradiction) $\mathcal F$ is empty, then
necessarily $\mathcal F_0$ must coincide with one of the four sets
$\mathcal F^{(i)}$. To fix ideas, say that
$\mathcal F_0=\mathcal F^{(NW)}$, i.e. everywhere in $\mathcal F_0$,
$\nabla\psi_{\mathtt w}$ is a non-trivial convex combination of
$\rho^{(W)}=(-n,0)$ and $\rho^{(N)}=(0,-n)$. Let $\gamma$ be the curve
along $\partial \mathcal F_0$ from point $A$ to point $B$, as in
Fig. \ref{fig:F2}, and let $t_p$ be the tangent vector at a point
$p\in\gamma$.
    \begin{figure}
	\begin{center}
		\includegraphics[height=6cm]{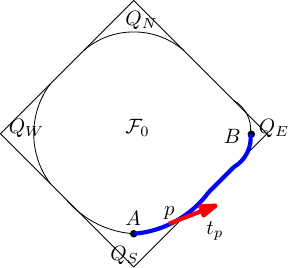}
		\caption{The curve $\gamma$ (in blue) and the tangent vector
                  ${t}_p$ at a point $p\in\gamma$.}
		\label{fig:F2}
	\end{center}
\end{figure}
From the definition of $Q_S,Q_W$ one has that the directional
derivative of $\psiw$ in direction $t_p$ equals $t_p\cdot g_p$, with
$g_p\in \ell^{(SE)}$. On the other hand, if $\gamma'$ is a curve from
$A$ to $B$ that runs slightly inside $\mathcal F_0$ at distance
$\delta$ from $\gamma$, we have that the directional derivative along
$\gamma'$ at a point $p'$ equals $t'_{p'}\cdot \nabla\psiw(p')$, with
$\nabla \psiw(p')\in \ell^{(NW)}$, because
$\mathcal F_0=\mathcal F^{(NW)}$ by assumption. Taking $\delta\to0$,
one easily sees that these two facts are not compatible with $\psiw$ being
continuous along $\gamma$. This proves that $\mathcal F$ is not empty.

Finally, the fact that $\mathcal F_R\ne\emptyset$ follows easily from
$\mathcal F\ne\emptyset$. In fact, if $\mathcal F_R$ were empty, then
$\nabla \psi_{\mathtt w}(x)$ would belong to $\mathcal S$ for every
$x\in \mathcal F$ and (because of Proposition \ref{prop:compactify})
it would actually take a constant value $\bar\rho$ on $\mathcal F$. If
$\mathcal F=Q$, this is a contradiction since the affine function with
slope $\bar \rho$  cannot match the boundary datum
$\psi_{\partial Q}$. If on the other hand
$Q\setminus \mathcal F\ne\emptyset$, then take a sequence of points
$x_i\in \mathcal F$ and a sequence $y_i \in Q\setminus \mathcal F$ that have the same limit in the interior of $Q$. One
has $\nabla\psiw(x_i)=\bar \rho$ while $\nabla\psi(y_i)\in \partial N(P)$, which contradicts Proposition \ref{prop:compactify}.
\end{proof}

\section{Properties of $v_{\mathtt w}(\rho)$} \label{sec:Properties}
\label{sec:proprieta}
We start with the following statement, whose proof is given below:
\begin{prop}
  \label{prop:vsmooth}
  The function $\rho\mapsto v_{\mathtt w}(\rho)$ is $C^\infty$ on $\mathcal R$.
\end{prop}

\begin{Remark}
  We know that the determinant of the Hessian matrix $J(x)$ of
  $\psi_{\mathtt w}$ is negative or zero on the rough region
  $\mathcal F_R$; if we knew that the inequality is everywhere strict,
  $C^\infty$ continuity of $v_{\mathtt w}(\cdot)$ would easily follow
  from formula \eqref{eq:implicit} below and from further derivation
  w.r.t. $\rho$. On the other hand, non-vanishing of $J(x)$ in the
  rough region is not a general property of macroscopic shapes of
  dimer models. For instance, for the dimer model on the honeycomb
  graph with uniform weights, one can verify from the explicit
  solution \cite{CLP98} that the macroscopic shape $\psi$ in a
  hexagonal domain has a Hessian with strictly negative determinant in
  the whole rough region, except at a single point (the center of the
  domain), where all entries of the Hessian matrix are zero. To overcome this problem, for the proof of Proposition \ref{prop:vsmooth} we will not
  rely directly on analytic properties of the limit shapes, bur rather
  on the definition \eqref{eq:prov} of the speed and on the properties
  of the dimer measure $\pi_{\rho,\mathtt w_j}$ under
  the dynamics $\{\mathtt w_j\}_{j\ge 0}$ of the edge weights (``spider move dynamics'').
\end{Remark}

From Proposition \ref{prop:vsmooth} and the formula \eqref{eq:chiev} for the speed, we deduce 
\begin{eqnarray}
  \label{eq:Dv}
  D v_{\mathtt w}(\rho)=-x_{\mathtt w}(\rho), \quad \rho\in\mathcal R.
 \end{eqnarray}
 By the way, this shows that $x_{\mathtt w}(\rho)$ is unique for
 $\rho$ in the rough region.  This formula also allows to prove that
 the speed is not $C^1$ at smooth slopes. Indeed, we know from Theorem
 \ref{prop:KO} that for every $\bar\rho\in \mathcal S$, there exists
 $x_{\mathtt w}(\bar\rho)$ in the interior of $Q$, where the slope of
 $\psiw$ is $\bar\rho$. Moreover, it is known \cite{Duse} that, since the boundary condition $\psi|_{\partial Q}$
 is ``natural'' (cf. footnote \ref{foot:natural}), the set
 $B_{\bar \rho}:=\{x\in Q: \nabla\psiw(x)=\bar \rho\}$ is a closed set
 with non-empty interior.  Letting $x\in \mathcal F_R$ approach
 different points of $B_{\bar\rho}$ (so that $\nabla\psiw(x)$
 approaches $\bar\rho$, by continuity of $x\mapsto \nabla\psiw(x)$),
 we see from \eqref{eq:Dv} that $D v_{\mathtt w}(\rho)$ does not have
 a unique limit as $\rho\to\bar \rho$.

From \eqref{eq:Dv} we see also that, for $\rho\in \mathcal R$,
\begin{eqnarray}
  \label{eq:implicit}
  D^2 v_{\mathtt w}(\rho) = - J(x_{\mathtt w}(\rho))^{-1},
\end{eqnarray}
where the $2\times 2$ Jacobian matrix $J(\cdot)$ is as in \eqref{eq:J}.
We already know that 
 $\det(J(x))\le0$, and the fact that the speed is $C^2$ means that the inequality is strict. In particular,
\begin{eqnarray}
  \det( D^2 v_{\mathtt w}(\rho))<0
\end{eqnarray}
as wished.

\begin{proof}
  [Proof of Proposition \ref{prop:vsmooth}]
  Let $f$ be an even face. From \eqref{eq:prov} and \eqref{eq:ch2} one
  has, with $\mathtt w\equiv \mathtt w_0$, \begin{eqnarray}
  \label{eq:cisiamoquasi}
           v_{\mathtt w}(\rho)=\lim_{k\to\infty}\frac1{4k} \sum_{j=0}^k\pi_{\rho,\mathtt w_j}[H(\eta)-V(\eta)].
  \end{eqnarray}
On the other hand, recall from \eqref{eq:hor} and \eqref{eq:vert} that $H(\eta),V(\eta)$ are sums of dimer indicator functions.
From the determinantal structure of the measures $\pi_{\rho,\mathtt w}$, one has an explicit expression for the probability that an edge $e$ is occupied.
Assume that the white endpoint of
$e$ is  in the fundamental domain $D_{m_1,m_2}$ (that is the
translation of $D_{0,0}$ by $2m_1 n$ in the horizontal direction and by
$2 m_2 n$ in the vertical one) and that, modulo this translation, it is equivalent to the white vertex $x$ of the fundamental domain $D_{0,0}$. Similarly, assume that the black endpoint is in $D_{\ell_1,\ell_2}$ and that it is equivalent to the black vertex $y$ in $D_{0,0}$.
Then,
  \begin{eqnarray}
    \label{eq:Ksol}
    \pi_{\rho,\mathtt w}[e\in \eta]= \mathbb K_{\mathtt w}(e)\mathbb K^{-1}_{\mathtt w}(e)
  \end{eqnarray}
  where
  $\mathbb K_{\mathtt w}(e)$ equals the $\mathtt w$-weight of $e$, times the complex unit $\mathrm{i}$ if the edge is vertical, while
  \begin{eqnarray}
      \label{eq:skoda}
	\mathbb K_{\mathtt w}^{-1}(e)=    \frac1{(2\pi \mathrm{i})^2}\int_{\substack{ |z|=e^{B_1} \\ |w|=e^{B_2}}} [K(z,w)^{-1}]_{y,x}  z^{m_1-\ell_1}w^{m_2-\ell_2}\frac{dz}{z} \frac{dw}{w}  .
    \end{eqnarray}
We recall that $K(z,w)$ is the $2 n^2\times 2n^2$ Kasteleyn
  matrix of the fundamental domain $D_{0,0}$ (recall Section
  \ref{sec:pwat}) and $B=B(\rho)= (B_1(\rho),B_2(\rho))$ is the value
  that realizes the supremum in \eqref{eq:sigma}.
  For
  $\rho=(\rho_1,\rho_2)\in \mathcal R$ the maximizer   is unique and the relation between
  $\rho$ and $B(\rho)$, through
  \begin{eqnarray}
    \label{eq:srBr}
\nabla \sigma(\rho)=B(\rho),    
  \end{eqnarray}
  is a $C^\infty$ diffeomorphism between $\mathcal R$ and
  $A(P)\subset \mathbb R^2$ (the amoeba of $P$, $A(P)$, defined as the
  image of the curve $P(z,w)=0$ in $\mathbb{C}^2$ under the map
  $(z,w) \mapsto (\log |z|, \log |w|)$)~\cite{Kenyon2003}.  We will
  prove:
  \begin{Lemma}
    \label{lemma:Cinft}
  The r.h.s. of \eqref{eq:skoda} is a $C^\infty$ function of $B$.  
  \end{Lemma}
   As a consequence, \eqref{eq:Ksol} and therefore
  the sum in \eqref{eq:cisiamoquasi}, for every fixed $k$, are $C^\infty$ functions of $\rho$.
  To conclude the proof of the proposition, we will prove:
  \begin{Lemma}
    \label{lemma:GK}
    Let $\mathtt w=\mathtt w_j$.
  The derivatives (of any order) of \eqref{eq:Ksol} w.r.t. $B$ can be bounded uniformly w.r.t. the index $j$.   
  \end{Lemma}
  The smoothness claim for $v_{\mathtt w}$ then easily follows from
  \eqref{eq:cisiamoquasi}.
\end{proof}

\begin{proof}[ 
  Proof of Lemma \ref{lemma:Cinft}]
  Assume without loss of generality (by translation invariance) that $\ell_1=\ell_2=0$.
  Write
  \begin{eqnarray}
    [K(z,w)^{-1}]_{y,x}=\frac{Q(z,w)}{P(z,w)}
  \end{eqnarray}
  with $P(z,w)=\det K(z,w)$ the characteristic polynomial and $Q(z,w)$
  (that is also a Laurent polynomial in $z,w$) the cofactor $(x,y)$ of
  $K(z,w)$, so that \eqref{eq:skoda} reduces to
\begin{eqnarray}
\frac{e^{B_1 m_1+B_2 m_2}}{(2\pi)^2}\int_0^{2\pi} d\theta \int_0^{2\pi} d \phi \frac{Q(e^{B_1+i\theta},e^{B_2+i\phi})}{P(e^{B_1+i\theta},e^{B_2+i\phi})}e^{i \theta m_1+i\phi m_2}.
\end{eqnarray}
The prefactor of the integral is smooth and will be dropped; also, we write
$\tilde Q$ for $Q\times e^{i \theta m_1+i\phi m_2}$.
  If $B=B(\rho)$ as in \eqref{eq:srBr} with $\rho\in\mathcal R$, it is known that $(\theta,\phi)\mapsto P(e^{B_1+i\theta},e^{B_2+i\phi})$ has two distinct simple zeros \cite{KOS03}, call them $(\theta^\omega,\phi^\omega),\omega=\pm$.
Write
\begin{eqnarray}
  P(e^{B_1+i\theta},e^{B_2+i\phi})=P^\omega_1+R^\omega:=a^\omega(\theta-\theta^\omega)+b^\omega(\phi-\phi^\omega)+R^\omega
\end{eqnarray}
where $P^\omega_1$ is the first-order Taylor expansion around $(\theta^\omega,\phi^\omega)$. The zeros $(\theta^\omega,\phi^\omega)$ and also $a^\omega,b^\omega$ are real analytic functions of $B_1,B_2$, and the ratio $a^\omega/b^\omega$ is not real.
Write
\begin{eqnarray}
1=f^+(\theta,\phi)+f^-(\theta,\phi)+(1-f^+(\theta,\phi)-f^-(\theta,\phi))  
\end{eqnarray}
where
\begin{eqnarray}
  f^\omega=\chi(|P^\omega_1|) 
\end{eqnarray}
and $\chi:\mathbb R\mapsto [0,1]$ is a $C^\infty$ function that equals $1$ (resp. $0$) when its argument is smaller than $\epsilon$ (resp. larger than $2\epsilon$), with $\epsilon$ sufficiently small so that the supports of $f^\pm$ are disjoint.
The integral of
\begin{eqnarray}
[1-f^+-f^-]\frac{\tilde Q}P
\end{eqnarray}
is $C^\infty$ w.r.t. $B$.  Now look at the integral of
$f^\omega \tilde Q/P$. Suppose we want to prove it is $C^k$ w.r.t $B$.
Write
\begin{eqnarray}
  \label{eq:Taylor}
\frac{\tilde Q}P=\frac{\tilde Q^\omega}{P^\omega_1}+\frac{\hat Q^\omega}{P^\omega_1}- \frac {\tilde Q R^\omega}  { P\,P^\omega_1},
\end{eqnarray}
with $\tilde Q^\omega:=\tilde Q(\theta^\omega,\phi^\omega)$ and $\hat Q^\omega:=\tilde Q-\tilde Q^\omega$.
Write 
\begin{eqnarray}
  a^\omega\theta+b^\omega\phi=X+i Y:=(\theta\Re (a^\omega)+\phi\Re( b^\omega) )+i (\theta\Im (a^\omega)+\phi\Im( b^\omega) ).
\end{eqnarray}
Since the ratio $a^\omega/b^\omega$ is not real, the Jacobian of the change of variables $(\theta,\phi)\leftrightarrow (X,Y)$ is non-singular.
One has then
\begin{eqnarray}
\nonumber \tilde Q^\omega \int_0^{2\pi} d\theta \int_0^{2\pi}d\phi\frac{f^\omega(\theta,\phi)}{P^\omega_1(\theta,\phi)}
  =\tilde Q^\omega\int_{\mathbb R^2} d\theta d\phi \frac{\chi(|a^\omega \theta+b^\omega\phi|) }{a^\omega\theta+b^\omega\phi}
 \\ =const\times \int_{\mathbb R^2} dX dY \frac{\chi(|X+i Y|)}{X+i Y}
\end{eqnarray}
which is zero by symmetry. 
Next look at
\begin{eqnarray}\int_0^{2\pi} d\theta \int_0^{2\pi}d\phi {f^\omega }\frac{\hat Q^\omega}{P^\omega_1}
  =\int_{\mathbb R^2} d\theta d\phi\,{\chi(|a^\omega \theta+b^\omega\phi|)}\frac{\hat Q^\omega (\theta+\theta^\omega,\phi+\phi^\omega)}{a^\omega\theta+b^\omega\phi}\\
  =  const\times \int_{\mathbb R^2}dX dY\chi(|X+iY|) \frac{\hat Q^\omega(X,Y)}{X+i Y}
  \label{eq:dadc}
\end{eqnarray}
where, with some abuse of notation, we write
\begin{eqnarray}
  \label{eq:abuse}
\hat Q^\omega(X,Y):=
\hat Q^\omega(\theta^\omega+\theta(X,Y),\phi^\omega+\phi(X,Y)).  
\end{eqnarray}
The
constant prefactor has a $C^\infty$ (in fact, real analytic)
dependence on $B$. Also, $\hat Q$ is a polynomial with real analytic
coefficients and it vanishes at least linearly when $(X,Y)$ tends to
zero. Then, it is easy to deduce that \eqref{eq:dadc} is a $C^\infty$
function of $B$.  
Finally, we look at
\begin{multline}
  \label{eq:finally}
 \int_0^{2\pi} d\theta \int_0^{2\pi}d\phi f^\omega  \frac{R^\omega\tilde Q}{P\,P^\omega_1}\\
 =
  const\times \int_{\mathbb R^2}dX dY \chi(|X+iY|)\frac{ \tilde Q (X,Y) R^\omega(X,Y)}{(X+i Y)(X+i Y+ R^\omega(X,Y))}.
\end{multline}
with the same convention as in \eqref{eq:abuse}.
Since $R^\omega$ is at least quadratic for
$X,Y$ close to zero, the  derivatives of order $k$ (w.r.t. the components of $B$) of the integrand are upper bounded by
\begin{eqnarray}
  \label{eq:dm}
  c(k)   
  \chi(|X+iY|)
\end{eqnarray}
uniformly for $B$ in compact sets of the amoeba $A(P)$.  The function \eqref{eq:dm} is  integrable and the claim of the Lemma easily follows.
\end{proof}

\begin{proof}[Proof of Lemma \ref{lemma:GK}]
  We have seen that for each choice of $\mathtt w$, the derivatives of \eqref{eq:Ksol} w.r.t. $B$ are bounded. Now we let $\mathtt w=\mathtt w_j$ and we need to show uniformity of the bounds w.r.t. $j$.
It is immediate to see that uniformity follows if all edge weights stay bounded away from $0$ and $\infty$, uniformly in $j$.
  
Let us recall that the probability measure $\pi_{\rho,\mathtt w}$
depends on the edge weights only modulo gauge transformations \cite[Sec. 3.2]{KenLectures}. That
is, if edge weights $\mathtt w$ are changed as
$\mathtt w(e)\mapsto \mathtt w(e) f(b)g(w)$, with $e$ the edge with
black/white endpoints $b/w$ and $f/g$ two non-vanishing functions
defined on black/white vertices, then the measure is unchanged. In the
($2n\times 2n$) periodic setting with fundamental domain $D_{0,0}$ as
in the present work, the knowledge of the edge weights modulo gauge is
equivalent to the knowledge of:
  \begin{enumerate}
  \item the ``face weights'': for each of the $4n^2$ faces $f$ of the
    fundamental domain $D_{0,0}$, one lets $\mathtt w(f)$ be the
    alternate product
    \[ \frac{\mathtt w(e_1)}{\mathtt w(e_2)} \frac{\mathtt
        w(e_3)}{\mathtt w(e_4)}
  \]
  with $e_1,\dots, e_4$ the four boundary edges of $f$ labeled
  cyclically clocwise, with $e_1$ chosen such that it is clockwise
  oriented from white to black endpoint.  Actually, the product of
  face weights over all faces gives $1$, so we need to know only
  $4n^2-1$ of them.
\item the ``magnetic coordinates'', i.e.  the alternate product $\mathtt W_1$ (resp. $\mathtt W_2$) of the
  weights of the edges belonging to a cycle on $D_{0,0}$
  with winding number $(1,0)$ (resp. $(0,1)$). 
\end{enumerate}
If the face weights, as well as $\mathtt W_1,\mathtt W_2$, are all
bounded away from $0$ and $+\infty$, then there exists a suitable
gauge such that edge weights are also all bounded away from $0$ and
$+\infty$.

When the weights $\mathtt w$ evolve along the sequence
$\{\mathtt w_j\}_{j\ge0}$ associated to the shuffling algorithm, the
magnetic coordinates $\mathtt W_1,\mathtt W_2$ stay constant
\cite{GK13}. This is related to the fact that the measure
$\pi_{\rho,\mathtt w_j}$ is mapped to $\pi_{\rho,\mathtt w_{j+1}}$ and
the slope $\rho$ is unchanged, recall Proposition
\ref{prop:misuramappata}.  On the other hand, the face weights do
change with $j$: in general they are not periodic in time but only
quasi-periodic, they stay in a compact set (that depends on the
initial weights $\mathtt w_0$) and they approach neither zero nor
infinity. This can be extracted from the classical integrability of
the dynamics of the face weights under the spider moves \cite{GK13}
(cf. also \cite{fock2015inverse} and \cite[Sec. 3]{Kenyon2003}). More explicitly, the spider move preserves the spectral curve and for positive-real-valued edge weights, the common level set of the Hamiltonians  is homeomorphic to a finite cover of the product of the compact ovals of the spectral curve.  
\end{proof}

{\bf Acknowledgements} We would like to thank Sanjay Ramassamy for
discussions on the Goncharov-Kenyon dynamical system, and Erik Duse
for sharing the results of \cite{Duse} before publication and for
discussions on the variational principle. We would also like to thank the referees for their careful reading and constructive comments. 
F.T.  was partially
supported by the CNRS PICS grant 151933, by ANR-15-CE40-0020-03 Grant
LSD and ANR-18-CE40-0033 Grant DIMERS. S.C acknowledges the support of EPSRC grant EP /T004290/1. 

\bibliographystyle{plain}
\bibliography{Biblio}

\end{document}